\documentclass[11pt]{article}

\usepackage{amssymb}
\usepackage[top=3cm,bottom=3cm,left=2cm,right=2cm]{geometry}
\usepackage{enumerate}
\usepackage{graphicx}
\usepackage{amsmath}
\usepackage{amsthm}
\usepackage[mathscr]{eucal}
\usepackage{mathrsfs}
\usepackage{graphicx}
\usepackage{fancybox}
\usepackage{color}
\def\a{\alpha}
\def\b{\beta}

\def\om{\omega}

\def\s{\sigma}

\def\wt{\widetilde}

\def\r{\eqref}
\def\Res{\mathop{\rm Res}}
\def\Im{\mathop{\rm Im}}
\def\Re{\text{\upshape Re\,}}
\def\Im{\text{\upshape Im\,}}
\def\rt{\wt R}

\def\cpm{\mathbb{C}_\pm}
\def\cp{\mathbb{C}_+}
\def\cm{\mathbb{C}_-}
\def\ra{\rightarrow}

\def\pt{\widetilde{\phi}}
\def\ps{\phi^*}
\def\gt {\wt g}
\def\et {\wt E}

\def\and{\mbox{ and } }
\def\xt{\wt \xi}
\def\xs{\xi^*}
\def\ent{\wt E_n(z) }

\def\xt{\wt \xi}
\def\xs{\xi^*}

\def\rt{\wt R}

\newtheorem{theorem}{Theorem}[section]
\newtheorem{defi}[theorem]{Definition}
\newtheorem{lemma}[theorem]{Lemma}
\newtheorem{prop}[theorem]{Proposition}

\def\Ai{\text{\upshape Ai}}
\def\Bi{\text{\upshape Bi}}

\def\pt{\widetilde\phi}

\def\ft{\widetilde f}
\def\fs{ f^* }

\numberwithin{equation}{section}

\def\hh{\notag \\ &=}
\def\hhh{\notag \\ &\quad}

\def\bb #1\ee {\begin{align}#1\end{align}}
\def\pb#1\pe{\begin{pmatrix}#1\end{pmatrix}}

\def\(#1\){\left(#1\right)}

\begin{document}

\title{Global Asymptotics of the Hahn Polynomials}
\author{Y. Lin$^{\,a,b}$ and R. Wong$^{\,b}$
\\
\hspace{.7cm}
 \hbox{\small \emph{ $^a$ Department of Mathematics, Zhongshan University, Guangzhou,
China}}
\\
 \hbox{\small
\emph{$^b$ Department of Mathematics, City University of Hong Kong, Tat Chee Avenue, Kowloon,
Hong Kong}
}
}
\date{}
\maketitle

\begin{abstract}
In this paper, we study the asymptotics of  the Hahn polynomials $Q_n(x;\a,\b,N)$
as the degree $n$ grows to infinity, when the parameters $\a$ and $\b$ are fixed and the ratio of $n/N=c$ is a constant in the interval $(0,1)$.
Uniform asymptotic formulas in terms of Airy
functions and elementary functions are obtained for $z$ in three overlapping regions,  which together cover the whole complex plane.
Our method is based on
a modified version of the  Riemann-Hilbert approach introduced by Deift and Zhou.
\end{abstract}

\section{Introduction}
For $\a, \beta>-1$, the Hahn polynomials are  explicitly given by
\begin{align}
  Q_n(x;\a,\b,N)&:={}_3F_2(-n,-x,n+\a+\b+1;-N,\a+1;1)
 \hh
     \sum_{k=0}^n \frac{  (-n)_k (-x)_k (n+\a+\b+1)_k  }{ (-N)_k (\a+1)_k \,k! };
 \label{Q}
\end{align}
see \cite[p.174]{BealsWong} and \cite{KarlinMcgregor}.
These polynomials
are orthogonal on the discrete set $\{ 0,1,\cdots,N \}$
with respect to the weight function
\begin{equation}
   \rho(x;\a,\b,N)=
  \frac{ (\a+1)_x }{ x! } \frac{ (\b+1)_{N-x} }{ (N-x)! },
\qquad
x=0,1,2,\cdots,N.
\label{rho}
\end{equation}
More precisely, we have
\begin{equation}
    \sum_{k=0}^{N}Q_n(k;\a,\b,N)Q_m(k;\a,\b,N)\rho(k;\a,\b,N)
=
  h_{N,n}  \delta_{n,m}
,
\qquad
    n,m=0,1,\cdots,N
,
\label{Q_ort}
\end{equation}
where
\begin{align}
  h_{N,n}=
  \frac {  (-1)^n (n+\a+\b+1)_{N+1} (\b+1)_n n!  }{ (2n+\a+\b+1)(\a+1)_n(-N)_n N! }
 ;
\end{align}
see \cite[p.\,462]{Handbook}.

Formula \eqref{Q_ort}  tells us that these polynomials are orthogonal on an unbounded interval as $n \rightarrow \infty$.
We now introduce a rescaling so that these polynomials become orthogonal on a bounded interval.
Let $X_N$ denote the set defined by
\begin{align}
    X_N:=\{x_{N,k} \}_{k=0}^{N-1}, \qquad \text{where\ \ } x_{N,k}:=\frac{k+1/2}{N}.
 \label{nodes}
\end{align}
The $x_{N,k}$'s are called \emph{nodes} and they all lie in the interval $(0,1)$. Also, let
\begin{equation}
  P_{N,n}(z):=Q_n(Nz-1/2;\a,\b,N-1)
\label{P_def}
\end{equation}
and
\begin{equation}
  w(z):=N^{-\a-\b}\a! \b!  \rho(Nz-1/2;\a,\b,N-1).
 \label{w_def}
\end{equation}
It is readily seen that the polynomials $P_{N,n}(z)$ are orthogonal on the nodes $x_{N,k}$ with respect to the weight $w_{N,k}:=w( x_{N,k} )$; that is,
\begin{equation}
    \sum_{k=0}^{N-1}P_{N,n}( x_{N,k} )P_{N,m}( x_{N,k} )  w_{N,k}
=
    h^*_{N,n}
    \delta_{n,m}
,
\label{P_ort}
\end{equation}
where  $h^*_{N,n}=N^{-\a-\b}\a!\b!h_{ N-1,n }$.

For further properties and applications of the Hahn polynomials, we refer to Beals and Wong \cite{BealsWong} and Karlin and Mcgregor \cite{KarlinMcgregor}.
In 1989, Sharapudinov \cite{Shar} studied the asymptotic behavior of $Q_n(x;\a,\b,N)$,
when the degree $n$ is substantially smaller than $N$. His main result is an asymptotic
formula for $Q_n(x;\a,\b,N)$ with $n=O(N^{1/2})$ as $n\rightarrow\infty$; which involves the Jacobi polynomial $P^{ (\b,\a) }_n(t)$,  where $t$ is related to $x$ via the formula $x=(N-1)(1+t)/2$.
Recently,
Baik et al. \cite{BaikBook} used the Riemann-Hilbert approach to investigate the asymptotics of
discrete orthogonal polynomials with respect to a general weight function.
Their results are very general, but it is difficult
to use them to write out explicit formulas for specific polynomials.
With regard to the Hahn polynomials, they only considered the case of varying parameters; that is,
$\a=NA$ and $\b=NB$, where $A$ and $B$ are fixed positive numbers.
Moreover, their results are more local in nature; that is, one needs more asymptotic formulas to describe the behavior of these polynomials in the complex plane.

The purpose of this paper is to study uniform asymptotic behavior of the Hahn
polynomials as $n\rightarrow \infty$, when the parameters $\a$ and $\b$ are fixed and the ratio $n/N$ is a constant $c\in(0,1)$. Our approach is based on a modified version of the Riemann-Hilbert method introduced by Deift and Zhou \cite{Deift}. This version has been
used successfully to obtain globally uniform asymptotic expansions of several
discrete orthogonal polynomials; see \cite{DaiWong}, \cite{OuWong} and
\cite{WangWong}.
The presentation of this paper is arranged as follows: In Section 2, we review some preliminaries, including weak asymptotics of the zero distribution and the formulation of the basic interpolation problem $Y(z)$. In Section 3,
we  transform this interpolation problem  into
an equivalent Riemann-Hilbert problem (RHP).
In Section 4, we give the matrix transformation $T(z)$ that normalizes the
RHP for $R(z)$ presented in Section 3 by using a function $\gt(z)$, which is related to the logarithmic potential (g-function) of the equilibrium measure.
The auxiliary functions
$\pt(z)$ and $\ps(z)$ are also studied in Section 4.  In Section 5, we factorize the
jump matrix in the RHP for $T(z)$.
The final results are presented in Sections 6 and 7, where we divide the whole complex plane into three regions and
in each region we derive a uniformly asymptotic formula. As a special case of our result, we give in Section 8 the asymptotic formula of $Q_n(x;\a,\b,N)$ for fixed values of $x$.


\section{Preliminaries}
\subsection{The equilibrium measure}

It is known that the zero distribution of discrete orthogonal polynomials is related to a constrained equilibrium problem for logarithmic potential with an external field $\varphi(x)$,
which is defined by the formula
\begin{align}
	\varphi(x):=V(x) + \int_{a_1}^{a_2} \ln|x-y| \rho^0(y)dy,
	\qquad x\in(a_1,a_2),
\end{align}
where $V(x)$ and $\rho^0(x)$ are real-analytic functions in a neighborhood of the closed interval $[a_1,a_2]$, and $\rho^0(x)$ is also a density function of the nodes; see \cite{DaiWong} and \cite[p.26]{BaikBook}.
In our case, $V(x)=1-x\log x -(1-x)\log(1-x)$, $\rho^0(x)=1$ and $[a_1, a_2]=[0,1]$.
To see how the function $V(x)$ is derived, we recall the nodes $x_{N,n}$ and the weights
$w_{N,n}=w(x_{N,n})$ given in \eqref{nodes} and \eqref{P_ort}, respectively, and write
\begin{align}
	w_{N,n}
 &= e^{ -N V_N( x_{N,n} ) } \prod^{N-1}_{ m=0 \atop m\neq n }
 	| x_{N,n}-x_{N,m} |^{-1}.
  \label{wNn}
\end{align}
From \eqref{wNn}, an explicit formula can be given for $V_N(x_{N,n})$. Indeed, since
$x_{N,n}=(n+\frac12)/N$ and
\begin{align*}
	\prod^{N-1}_{ m=0 \atop m\neq n }
 	| x_{N,n}-x_{N,m} |^{-1}
 =	\frac{ N^{N-1} }{ n! (N-n-1)! },
\end{align*}
we have
\begin{align*}
	V_N(x_{N,n})= -\frac1N \log w(x_{N,n})+
		\frac1N \log \frac{ N^{N-1} }{ n! (N-n-1)! }.
\end{align*}
Since $n=N x_{N,n}-\frac12$, by replacing $x_{N,n}$ and $n$ by $x$ and $Nx-\frac12$, respectively, in the above equation, we obtain
\begin{align}
	V_N(x)=-\frac1N \log w(x)+
		\frac1N \log \frac{ N^{N-1} }{ (Nx-\frac12)! (N-Nx-\frac12)! }
 \label{Vx}
\end{align}
for $x\in(0,1)$. Using Stirling's formula, it can be shown that
\begin{align*}
	\frac1N \log \frac{ N^{N-1} }{ (Nx-\frac12)! (N-Nx-\frac12)! }
 =	1-x\log x -(1-x) \log(1-x)+O\(\frac1N\)
\end{align*}
as $N\rightarrow\infty$. From \eqref{w_def}, we also have
\begin{align*}
	 w(x)&=
	  N^{-\a-\b}
 \frac{ \Gamma(Nx+\a+\frac12)}{  \Gamma(Nx+\frac12)}
 \cdot
 \frac{\Gamma( N(1-x)+\b+\frac12  ) }{\Gamma(N(1-x)+\frac12)    }
\\
 &\sim
	x^{\a}(1-x)^{\b}
	\qquad\qquad \mbox{as } N\rightarrow\infty.
\end{align*}
Thus, it follows from \eqref{Vx} that
\begin{align}
	V_N(x) =
	1-x\log x -(1-x)\log(1-x)+O\(\frac1N\)
 \label{V}
\end{align}
as $N\rightarrow \infty$. This approximation formula suggests the definition of the function
\begin{align}
	V(x)=1-x\log x -(1-x)\log(1-x),
\end{align}
and we have
\begin{align}
	V_N(x)=V(x)+O\(\frac1N\).
\end{align}
Observe that in our case, we also have $\varphi(x)=0$.

The measure $\mu(x)dx$ that minimizes the following energy functional
\begin{align}
	E_c[\mu]=-c \int_0^1 \int_0^1 \ln |x-y| \mu(x) dx \mu(y) dy +
	\int_0^1 \varphi(x) \mu(x) dx,
\end{align}
where $\mu(x)$ satisfies the upper and lower constraints
\begin{align}
	0\leq \mu(x) \leq 1/c
 \label{mu_con}
\end{align}
and the normalization condition
\begin{align}
	\int_0^1 \mu(x) dx =1,
 \label{mu_nor}
\end{align}
is called the {\it equilibrium measure}.
It follows from (\ref{mu_con}) that there is an upper bound for $\mu(x)$, which does not appear in the continuous case. This fact is a crucial difference between
discrete orthogonal polynomials and continuous orthogonal polynomials.
The equilibrium measure $\mu(x)dx$ divides the
interval of orthogonality into subintervals of three possible types:
(1) achieving the lower constraint; (2) attaining the upper constraint;
(3) not reaching the upper or the lower constraints. We call the open intervals of type (1)-(3), Voids, Saturated Regions and Bands, respectively. For reference, we refer to
\cite[p.19]{BaikBook}.

In the existing literature, the equilibrium measure is usually obtained by solving a minimization problem of a certain quadratic
functional (cf. \cite{BaikBook}, \cite{DaiWong}). Here, we prefer to use the method introduced by Kuijlaars and Van Assche \cite{Arno}.

Following \cite{Arno}, one can construct the
equilibrium measure corresponding to the Hahn polynomials. Recall $c:=n/N$, and let
\begin{equation}
    a:=\frac12 -\frac12 \sqrt{1-c^2},
\qquad \qquad
    b:=\frac12+\frac12 \sqrt{1-c^2}.
\label{MRS}
\end{equation}
The density function is given by
\begin{equation}
  \mu(x)=\left\{\begin{array}{ll}
  \dfrac{2}{\pi c}  \arcsin( \dfrac{c}{2\sqrt{ x-x^2 } } ), &  x\in(a,b),\\
\\
  \dfrac1c,  & x\in[0,a] \cup [b,1].
  \end{array}\right.
\label{mu}
\end{equation}
Observe that the equilibrium measure of the Hahn polynomials corresponds to the saturated-band-saturated configuration defined in \cite{BaikBook}, and we also note that
\begin{align}
	a+b=1,
	\qquad \qquad
	a\cdot b=c^2/4.
 \label{ab}
\end{align}

\subsection{The basic interpolation problem}

We begin with the  basic interpolation problem
 for the Hahn polynomials.
Note that the leading coefficient of $P_{N,n}(z)$ is
\begin{equation}
k_{N,n}:=\frac{ N^n(n+\a+\b+1)_n }{ (\a+1)_n (-N+1)_n }.
\end{equation}
Hence,
the monic Hahn polynomials $\pi_ { N,n}(z)$ are given by
\begin{equation}
  \pi_{N,n}(z)=\frac {1}{k_{N,n}} P_{N,n}(z)
.
\label{pi_def}
\end{equation}
Clearly, they satisfy the new orthogonality relation
\begin{equation}
     \sum_{k=0}^{N-1}\pi_{N,n}(   x_{N,k}  )\pi_{N,m}(x_{N,k} )w_{N,k}
=
    \delta_{n,m}/\gamma_{N,n}^2
,
\label{pi_orthogonality}
\end{equation}
where
$
    \gamma_{N,n}^2
    =
    k_{N,n}^2/h^*_{N,n}
$.

Following \cite{BaikBook}, we construct the interpolation problem
 for a $2\times2$ matrix-value function $Y(z)$ with the properties:
\begin{enumerate}
\item[($\,Y_a$)] $Y(z)$ is analytic for $z\in \mathbb{C}\setminus X_N$;
\item[($\,Y_b$)] at each $x_{N,k} \in X_N $, the second column of $Y$ has a simple pole where the residue is
\begin{equation}
    \Res\limits_{z=x_{N,k}} Y(z)=\lim_{z\rightarrow x_{N,k} } Y(z)
    \begin{pmatrix} 0 & w_{N,k}  \\ 0 & 0 \end{pmatrix};
 \label{Y_res}
\end{equation}
    \item[($\,Y_c$)] as $z\rightarrow \infty$,
$$
    Y(z)
=
    \left(  I + O\left(\frac1z \right)  \right)
    \begin{pmatrix}
    z^n  & 0\\
    0 & z^{-n}
    \end{pmatrix}.
$$
\end{enumerate}
By the well-known theorem of Fokas, Its and Kitaev \cite{FokasIts}
(see also Baik et al. \cite{BaikBook}),
we have
\begin{theorem}
The unique solution to the above interpolation problem is given by
\begin{equation}
    Y(z)
:=
    \begin{pmatrix}
    \pi_{N,n}(z)  &  \sum\limits_{k=0}^{N-1}\dfrac{ \pi_{N,n}(x_{  N,k	})w_{N,k} }{z-x_{  N,k	}}\\
    \\
    \gamma_{N,n-1}^2\pi_{N,n-1}(z)  &  \sum\limits_{k=0}^{N-1}
    \dfrac{\gamma_{N,n-1}^2\pi_{N,n-1}(x_{  N,k	})w_{N,k} }
    { z-x_{  N,k	} }
    \end{pmatrix}
.
\label{Y}
\end{equation}
\end{theorem}

\section{Riemann-Hilbert Problem}

In this section,
we will  make  a sequence of transformations:
$
	Y\rightarrow H\rightarrow R,
$
and transform the basic interpolation problem into
an equivalent Riemann-Hilbert problem (RHP).


First,  we want to
remove the saturated regions.
Following Baik et al. \cite{BaikBook},
we  introduce the first transformation
\begin{align}
    H(z)
&:=
     Y(z)
        \begin{pmatrix}
         \prod\limits_{j=0}^{N-1}(z-x_{N,j})^{-1} & 0 \\
        0 & \prod\limits_{j=0}^{N-1}(z-x_{N,j})
        \end{pmatrix}
.
\label{H}
\end{align}
It is readily verified that
$H(z)$ is a solution of the
interpolation problem:
\begin{enumerate}
\item[($\,H_a$)] $H(z)$ is analytic for $z\in \mathbb{C}\setminus X_N$;
\item[($\,H_b$)] at each $x_{N,k}\in X_N$, the second column of $H(z)$ is analytic and
the first column of $H(z)$ has a simple pole where the residue is
\begin{equation}
\Res\limits_{z=x_{N,k}} H(z)=\lim_{z\rightarrow
x_{N,k}} H(z)
\begin{pmatrix} 0 & 0\  \\ \\
 \dfrac{  1  }{ w_{N,k}   }
\prod\limits^{N-1}_{j=0\atop j\neq k}(x_{N,k}-x_{N,j}  )^{-2}  & 0\  \end{pmatrix};
 \label{H_res}
\end{equation}
\item[($\,H_c$)] as $z\rightarrow \infty$,
$$
H(z)=\left(  I + O\left(\frac1z\right)  \right)
\begin{pmatrix}
z^{n-N}  & 0\\
0 & z^{-n+N}
\end{pmatrix}.
$$
\end{enumerate}

Next, we shall employ an idea of Ou and Wong \cite{OuWong}   to remove the poles in $H(z)$, and transform this interpolation problem  into an equivalent RHP;
see also  \cite{WangWong}.
Let $\delta>0$ be a sufficiently small number, and we define
\begin{equation}
    R(z)
:=  H(z)
    \begin{pmatrix}
    1 & \ 0 \ \   \\\\
    \dfrac{\mp ie^{ \pm N \pi iz  }  \cos(N\pi z) }{   N\pi w(z)  \prod\limits_{j=0}^{N-1}(z-x_{N,j})^2  } & \ 1 \
    \end{pmatrix}
\label{R_H_ort}
\end{equation}
for $z\in\Omega_{\pm}$, and
\begin{equation}
R(z):=H(z)
\label{R_H}
\end{equation}
for $z\notin\Omega_{\pm}$,
where the domains $\Omega_{\pm}$ are shown in Figure \ref{Figure_R}.
Let $\Sigma_+$ be the boundary of $\Omega_+$ in the upper half-plane, and $\Sigma_-$ be the mirror image of $\Sigma_+$ in the lower half-plane. For the contour $\Sigma=(0,1)\cup\Sigma_\pm$, see also Figure
\ref{Figure_R}.
\begin{figure}[htbp]
\begin{center}
\includegraphics[scale=1]{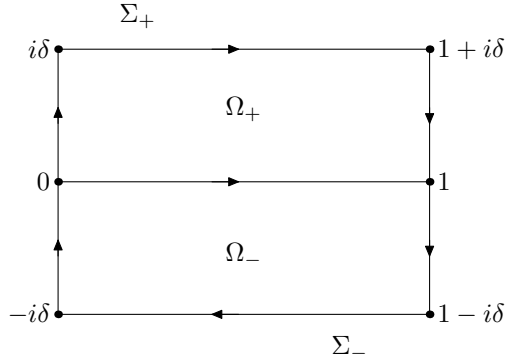}
\caption{The domains $\Omega_{\pm}$ and the contour $\Sigma$.}
\label{Figure_R}
\end{center}
\end{figure}

\begin{lemma}
For each $x_{N,k}\in X_N$,  the singularity of $R(z)$ at $x_{N,k}$ is removable; that is,
$\Res\limits_{z=x_{N,k}}R(z)=0$.
 \label{R_lem}
\end{lemma}
\begin{proof}
For  $z\in \Omega_{\pm}$, it follows from (\ref{R_H_ort}) that
\begin{equation}
R_{11}(z)=H_{11}(z)+H_{12}(z)
\dfrac{\mp ie^{ \pm N \pi iz  }  \cos(N\pi z) }{   N\pi w(z)  \prod_{j=0}^{N-1}( z-x_{N,j} )^2  }
,
\qquad
R_{12}(z)=H_{12}(z)
.
\label{R_11}
\end{equation}
Since $H_{12}(z)$ is  analytic by ($H_b$),
the residue of $R_{12}(z)$ at $x_{N,k}$ is zero.
From \r{H_res}, we also have
\begin{equation}
\Res\limits_{z=x_{N,k}}H_{11}(z)=H_{12}(x_{N,k})
\frac{1}{w_{N,k}}
\prod\limits^{N-1}_{j=0\atop j\neq k}(x_{N,k}-x_{N,j}  )^{-2}.
\label{H_11}
\end{equation}
On the other hand, it is readily seen that
\begin{equation}
    \Res\limits_{z=x_{N,k}} H_{12}(z)
    \dfrac{\mp ie^{ \pm N \pi  iz  }  \cos(N\pi z) }{   N\pi  w(z) \prod_{j=0}^{N-1}( z-x_{N,j} )^2  }
=
    -H_{12}(x_{N,k})
    \frac{1}{w_{N,k}}
    \prod\limits^{N-1}_{j=0\atop j\neq k}(x_{N,k}-x_{N,j}  )^{-2}.
\label{H_12}
\end{equation}
Thus, applying (\ref{H_11}) and (\ref{H_12}) to (\ref{R_11}) shows
that the residue of $R_{11}(z)$ at $z=x_{N,k}$ is zero.
The entries in the second row of the matrix $R(z)$ can be studied similarly. This completes the proof of the lemma.
\end{proof}
Note that
this transformation makes $R_+(x)$ and $R_-(x)$
continuous on the interval (0,1).
As a consequence, we
have created several jump discontinuities on the contour $\Sigma$ in the complex plane.

It is easily verified that $R(z)$ is a solution of the following RHP:
\begin{enumerate}
\item[($R_a$)] $R(z)$ is analytic for $z\in \mathbb{C}\setminus \Sigma$;
\item[($R_b$)] the jump conditions on the contour $\Sigma$: for $x\in (0,1)$,
\begin{equation}
    R_+(x)
=
    R_-(x)
    \begin{pmatrix}
    1 & 0\\
    r(x) & 1
    \end{pmatrix}
,
\label{R_real}
\end{equation}
where
\begin{equation}
r(x)= \frac{4\cos^2(N\pi x)}{ c_1 w(x)
 \prod^{N-1}_{j=0}(x-x_{N,j}  )^2 },
 \qquad\qquad c_1:=2N\pi i;
\label{r}
\end{equation}
for $z\in \Sigma_{\pm}$,
\begin{equation}
R_+(z)=R_-(z)
\begin{pmatrix}
1 & 0\\
\widetilde{r}_{\pm}(z)  & 1
\end{pmatrix},
\label{R:jum:S}
\end{equation}
where
\begin{equation}
\widetilde{r}_{\pm}(z)=\frac{\mp 2 e^{ \pm N\pi iz }
\cos(N\pi z )}{  c_1 w(z) \prod^{N-1}_{j=0}(z-x_{N,j}  )^2  } ;
\label{rtilde}
\end{equation}
\item[($R_c$)] as $z\rightarrow \infty$,
\begin{equation}
R(z)=\left(  I + O\left(\frac1z\right)  \right)
\begin{pmatrix}
z^{n-N}  & 0\\
0 & z^{-n+N}
\end{pmatrix}.
\label{R_inf}
\end{equation}
\end{enumerate}


\section{The  Auxiliary Functions}
To normalize the condition ($R_c$) in the RHP for $R(z)$, we need some
auxiliary functions.
\begin{defi}
The $g$-function is the logarithmic potential  defined by
\begin{equation}
g(z):=\int_{0}^{1}  \log(z-s) \mu(s)\,ds, \qquad \qquad z\in \mathbb{C} \setminus (-\infty,1],
\label{g_def}
\end{equation}
where $\mu(x)$ is the density function given in \r{mu},
and the so-called $\phi$-function is defined by
\begin{equation}
\phi(z):=l /2 -g(z),
\qquad \qquad z\in \mathbb{C}\setminus (-\infty,1],
\label{phi_def}
\end{equation}
where $l:=2\int_0^1\log|a-s|\mu(s) ds$ is called the Lagrange multiplier, a being the constant
given in \r{MRS}.
\end{defi}
Now, we  define
\begin{align}
	\wt g(z):= g(z) -\frac 1c\varphi(z),
 \label{gt_def}
\end{align}
where
\begin{align}
	\varphi(z)= \int_0^1 \log(z-s) ds.
	\label{vp}
\end{align}
In \r{g_def} and \r{phi_def}, we view $\log(x-s)$ as an analytic function of the variable $s$ with a branch cut along $(x,+\infty)$. Since
$$
	\log_\pm(x-s) =\log|x-s|\pm \pi i, \qquad
	s\in(x,+\infty),
$$
we obtain
\begin{align*}
	g_\pm(x)&=\int_0^x   \log(x-s) \mu(s) ds + \int_x^1 (  \log|x-s|\pm \pi i  ) \mu(s)ds
	\hh
		\int_0^1 \log|x-s| \mu(s) ds \pm \pi i  \int_x^1 \mu(s)ds.
\end{align*}
Similarly, we have
$$
	\varphi_\pm (x) =\int_0^x \log|x-s| ds +\int_x^1 \log_\pm (x-s) ds
	=\int_0^1 \log|x-s| ds \pm \pi i (1-x).
$$
From \eqref{gt_def},  it follows that
\begin{align}
	\gt_\pm(x)=\int_0^1 \log|x-s| (\mu(s)-\frac1c ) ds\pm \pi i\int_x^1 (\mu(s)-\frac1c) ds,
	\qquad x\in(0,1).
 \label{g_pm}
\end{align}
Moreover, it can be easily verified  that the $\gt$-function satisfies the jump conditions
\begin{align}
	&\gt_+(x)-\gt_-(x)=2\pi i(1-\frac1c),
	\qquad x\in(-\infty,0)
 \label{g_jum1}
\end{align}
and
\begin{equation}
	\gt_+(x)-\gt_-(x)=
	\left\{
	\begin{aligned}
	&2\pi i (1-\frac1c ), &&x\in(0,a), \\
&{	2\pi i \int_x^b }( \mu(s)-\frac1c) ds, && x\in(a,b) , \\
	&0, && x\in(b,1).
	\end{aligned}
	\right.
 \label{g_jum2}
\end{equation}
On account of (\ref{g_def}), (\ref{g_jum1}) and \eqref{g_jum2}, one readily sees that
$e^{n\gt(z)}$ can be analytically extended to $\mathbb{C}\setminus [a,b]$
and
\begin{align}
	e^{n\gt(z)} =z^{n-N} [ 1+ O(\frac1z) ]
	\qquad \text{as\ } z\rightarrow \infty.
	\label{g_z}
\end{align}
Hence, by adopting the convention that $\s_3$ denotes the Pauli matrix
\begin{align}
	\s_3=
	\begin{pmatrix}
	1 & 0 \\
	0 & -1
	\end{pmatrix},
\end{align}
we introduce the transformation
\begin{equation}
	T(z):=(c_1e^{ nl })^{-\s_3/2} R(z)
	e^{-n\wt g(z)\s_3}
	  (c_1e^{ nl })^{\s_3/2}
 \label{T}
\end{equation}
to normalize  the behavior of $R(z)$
at $z=\infty$, where $c_1$ is given in \eqref{r}.

 As we have  mentioned before,
 the density function $\mu(x)$ reaches its upper constraint near the endpoints of the interval of the orthogonality;
see (\ref{mu}).
Furthermore, since $\mu(x)$ is not differentiable  at the point $x=a$, the function $\phi(x)$ is  not analytic in the neighborhood of $x=a$, and we can not construct our global parametrix (such as Airy parametrix) by using $\phi(z)$ in the region which includes the critical point $a$.
Hence, for our future analysis, a few more auxiliary functions are needed.

Note by \eqref{g_jum2} that ${\gt(z)}$ is analytic for $z\in(b,1)$.
This together with (4.5)   evokes us to  consider a modified measure $\mu^*(x)=\mu(x)-\frac1c$.
Furthermore, we want to find a complex function $\nu(z)$ which  satisfies the requirement
\begin{align}
	\nu_\pm(x)
	&=
	\pm \pi i (\mu(x)-\frac1c)
\end{align}
for $x\in(a,b)$.


Let
\begin{align}
	\nu(z):=\frac2c \log[c/2 - \sqrt{(z-a)(z-b)}  ]-
	\frac1c \log(z-z^2),
 \label{nu}
\end{align}
where  $\sqrt{(z-a)(z-b)}$ is analytic in
$\mathbb{C}\setminus[a,b]$ and behaves like $z$ as $z\rightarrow +\infty$,
and  where $\log(z-z^2)$ is an analytic function with branch cuts along
$(-\infty,0]\cup[1,+\infty)$.

In view of the identities
\begin{align}
	 \arccos x
	=- i \log( x+ i \sqrt{1-x^2} ),
	\qquad
	-\arccos x
	=- i \log( x- i \sqrt{1-x^2} )
 \label{arccos}
\end{align}
for $x\in[-1,1]$,   we have
\begin{align}
	\nu_\pm(x)
	&=\mp   \frac2c i\arccos(\frac{c}{2 \sqrt{x-x^2} } ),
	 \qquad  x\in(a,b).
 \label{nu_pm_pre}
\end{align}
On the other hand, the measure $\mu(x)-1/c$ can be rewritten as
\begin{align}
	\mu(x)-\frac1c
	=- \frac{2}{\pi c}  \arccos(\frac{  c  }{  2\sqrt{x-x^2}  }),
	\qquad x\in(a,b).
 \label{mus}
\end{align}
From \eqref{nu_pm_pre} and \eqref{mus},  it clearly follows that
\begin{align}
	\nu_\pm(x)
	&=
	\pm \pi i (\mu(x)-\frac1c)
 \label{nu_pm},
 \qquad x\in(a,b).
\end{align}
Now, we are ready to introduce the auxiliary function
\begin{align}
	\pt(z)&:=\int_a^z \nu(s) ds,
	\qquad z\in\mathbb{C}\setminus (-\infty,0]\cup[a,+\infty),
 \label{pt_def}
\end{align}
where the path of integration from $a$ to $z$ lies entirely in the region
$z\in\mathbb{C}\setminus (-\infty,0]\cup[a,+\infty)$, except for the initial point $a$.

Similarly, we set
\begin{equation}
	\ps(z):=\int_b^z \nu(s) ds,
	\qquad z\in\mathbb{C}\setminus (-\infty,b]\cup[1,+\infty),
 \label{ps_def}
\end{equation}
where the path of integration from $b$ to $z$ lies entirely in the region
$z\in\mathbb{C}\setminus (-\infty,b]\cup[1,+\infty)$, except for the initial point $b$.

The functions $\pt(z)$ and $\ps(z)$ defined above play an important role
in our argument, and the following are some of their properties.

\begin{prop}
The mapping properties of $\pt(z)$ are given by:
\begin{align}
	\Re \pt_\pm (x) <  \pt(0), \quad
	\Im \pt_\pm (x)=\mp \frac1 c \pi x, \quad
	x\in(-\infty,0);
\label{pt_map_min}
\\
	\pt(x)<0, \quad
	\Im \pt(x)=0, \quad
	x\in[0,a);
\label{pt_map_0}
\\
	\pt(a)=0, \quad
	 \pt_\pm(b)=\pm (1-\frac1c)\pi i ;
\label{pt_map_con}
\\
	\Re \pt_\pm(x)=0, \quad
	\arg \pt_\pm(x)=\mp \frac\pi 2, \quad
	x\in(a,b);
\label{pt_map_ab}
\\
	\Re \pt_\pm(x)<0, \quad
	\Im \pt_\pm(x)=\pm(1-\frac1c)\pi, \quad
	x\in(b,1];
\label{pt_map_b}
\\
	\Re \pt_\pm (x) <  \Re \pt_\pm(1), \quad
	\Im \pt_\pm (x)=\pm (1-\frac1c x) \pi , \quad
	x\in(1,+\infty).
\label{pt_map_1}
\end{align}
\label{pro_pt_map}
\end{prop}
\begin{proof}

From the definition in (\ref{nu}), we note that $\nu(z)$ is  analytic in $\mathbb{C}\setminus (-\infty,0]\cup[a,b]
\cup[1,+\infty)$.
When $x\in(-\infty,0)$, it is readily seen that
\begin{align}
	\nu_\pm(x)
 =
 	\nu_1(x)  \mp \frac1c \pi i,
 \label{nu_pm1}
\end{align}
where
\begin{align}
	\nu_1(x)=\frac2c\log[c/2 +\sqrt{ (a-x)(b-x) } ]
 	-\frac1c \log(x^2-x).
 \label{nu1}
\end{align}
From \r{pt_def} and (\ref{nu_pm1}), we have
\begin{align}
	\pt_\pm(x)
	&= \int_a^0 \nu(s) ds+ \int_0^x \nu_\pm(s) ds
 \hh
 	\pt(0)+ \int_0^x
 	\nu_1(s)
 	ds
 	\mp \frac1c \pi x i.
 \label{pt_p_min}
\end{align}
It is easily verified that
$
	\nu_1(s)>0
$
for $s\in(-\infty,0)$, and
hence
(\ref{pt_map_min}) holds by  (\ref{pt_p_min}).

For $x\in[0,a)$,
$\nu(x)$ is analytic and it can be shown from \r{nu} that
$\nu(x)$ is real and positive. Thus, from \r{pt_def} it follows that
\begin{align}
 \Im \pt(x)=0
 \qquad\ \  \mbox{and} \qquad\ \
	\pt(x)
	&<0;
 \label{pt_p_0}
\end{align}
i.e., \eqref{pt_map_0} holds.

On the other hand, $\pt(a)=0$ follows immediately from the definition in (\ref{pt_def}).
Note that  $\nu_\pm(x)=\pm \pi i (\mu(x)-\frac1c)$ for $x\in(a,b)$; see (\ref{nu_pm}).  Hence,   (\ref{pt_map_con}) holds by (\ref{pt_def}).
If  $x\in(a,b)$, we also have from (\ref{nu_pm})
\begin{align}
	\pt_\pm(x)&=\pm \pi i \int_a^x (\mu(s)-\frac1c) ds.
	\label{t3}
\end{align}
Note that  $\mu(s)-\frac1c<0$ for $s\in(a,b)$ by  (\ref{mu}),  and hence (\ref{t3}) gives (\ref{pt_map_ab}).

For $x\in(b,1]$, it follows from (\ref{pt_def}) that
\begin{align}
	\pt_\pm(x)&= \int_a^b  \nu_\pm(s)   ds +\int_b^x  \nu(s)   ds
	\hh \pt_\pm(b) +\int_b^x  \nu(s)   ds .
 \label{t4}
\end{align}
By a change of variable $s=1-\zeta$, it can be verified that
\begin{align}
	\int_b^x  \nu(s)   ds
	&= \int_a^{1-x} -\nu(1-\zeta) d\zeta
 \hh
 	\int_a^{1-x} \nu(\zeta) d\zeta
\hh
	\pt(1-x).
 \label{t6}
\end{align}
Applying  (\ref{pt_map_con}) and (\ref{t6}) to  (\ref{t4}) yields
\begin{align}
	\pt_\pm(x)
	&=
 	\pt(1-x) \pm (1-\frac1c)\pi i.
 \label{t5}
\end{align}
Thus, (\ref{pt_map_b}) holds by (\ref{pt_map_0}) and (\ref{t5}).

For $x\in(1,+\infty)$, we have
\begin{align}
	\nu_\pm(x)
 =
 	\nu_2(x)   \mp \frac1c \pi i,
 \label{nu_pm2}
\end{align}
where
\begin{align}
	\nu_2(x)=\frac2c\log[\sqrt{ (x-a)(x-b) }-c/2 ]
 	-\frac1c \log(x^2-x).
 \label{nu2}
\end{align}
From (\ref{pt_map_b}) and (\ref{nu_pm2}), it follows that
\begin{align}
	\pt_\pm(x)
	&=
 	\int_a^1 \nu_\pm(s) ds +\int_1^x \nu_\pm(s) ds
 \hh
 	\pt_\pm(1) +\int_1^x \nu_2(s) ds  \mp \frac1c (x-1)\pi i
 \hh
 	\Re \pt_\pm(1)+\int_1^x \nu_2(s) ds \pm (1-\frac1c x) \pi i.
 \label{t9}
\end{align}
Note that  $\nu_2(s)<0$ for $s\in(1,+\infty)$.
Thus, $\int_1^x \nu_2(s) ds<0$ and (\ref{pt_map_1}) holds.

\end{proof}

\begin{figure}[!tpb]
\begin{minipage}[t]{0.45\linewidth}
\centering
\includegraphics[scale=1]{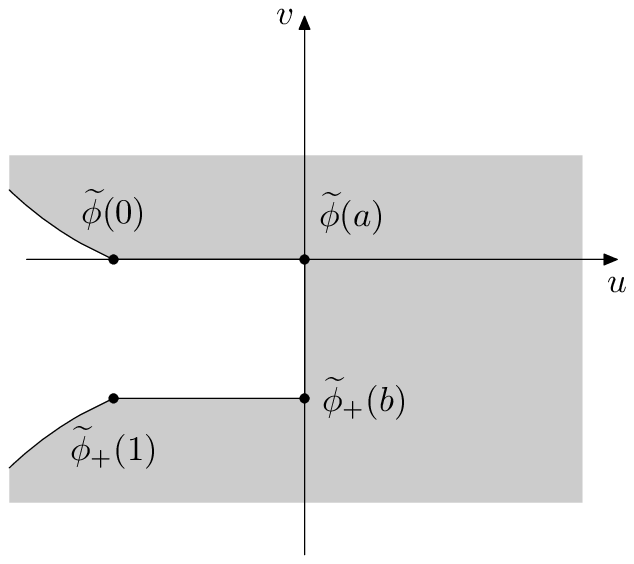}
\\
\medskip
(a) The upper half $z$-plane under $\pt(z)$
\end{minipage}
\hfill
\begin{minipage}[t]{0.45\linewidth}
\centering
\includegraphics[scale=1]{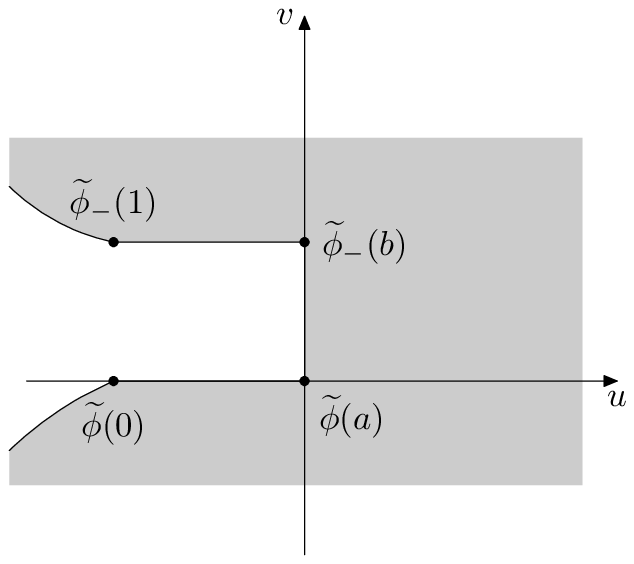}
\\
\medskip
(b) The lower half $z$-plane under $\pt(z)$
\end{minipage}
\caption{The image of the $z$-plane  under the mapping  $\pt(z)$.}
\label{pt_fig}
\end{figure}
A geometric interpretation of the results established above is given in Figure
\ref{pt_fig}.
Furthermore, if we let $\mathbb{C}_+=\{z: \text{Im\,}z>0\}$
and $\mathbb{C}_-=\{z: \text{Im\,}z<0\}$, then for $z\in\mathbb{C}_\pm$
we have from \r{nu_pm} and \r{pt_def}
\begin{align}
	\ps(z)
 =
 	\int_b^z
 	\nu(s) ds
=
 	\int_a^z \nu(s) ds \pm \pi i \int_b^a \(  \mu(s) -\frac1c  \) ds
=
 	\pt(z)\pm \pi i (\frac1c -1)
 .
 \label{ps_pt}
\end{align}
The results in the following proposition can now
be proved either directly from definition (\ref{ps_def}), or by using \r{ps_pt} and Proposition \ref{pro_pt_map}.

\begin{prop}
The mapping properties of $\ps(z)$ are given by:
\begin{align}
	\Re \ps_\pm (x) <  \ps(1), \quad
	\Im \ps_\pm (x)=\pm \frac1 c  (1-x) \pi, \quad
	x\in(1,+\infty);
\label{ps_map_1}
\\
	\ps(x)<0, \quad
	\Im \ps(x)=0, \quad
	x\in(b,1];
\label{ps_map_b}
\\
	\ps(b)=0, \quad
	 \ps_\pm(a)=\pm (\frac1c-1)\pi i ;
\label{ps_map_con}
\\
	\Re \ps_\pm(x)=0, \quad
	\arg \ps_\pm(x)=\pm \frac\pi 2, \quad
	x\in(a,b);
\label{ps_map_ab}
\\
	\Re \ps_\pm(x)<0, \quad
	\Im \ps_\pm(x)=\pm (\frac1c-1)\pi, \quad
	x\in[0,a);
\label{ps_map_0}
\\
	\Re \ps_\pm (x) <  \Re \ps_\pm(0), \quad
	\Im \ps_\pm (x)=\pm (\frac{1-c}{c}- \frac1c  x )\pi, \quad
	x\in(-\infty,0).
\label{ps_map_min}
\end{align}
\label{pro_ps_map}
\end{prop}
The images of the upper (lower) half of $z$-plane under the mapping $\ps(z)$ is
depicted in Figure \ref{ps_fig}.
The next proposition shows the relationship between the $\phi$-function and
the $\wt\phi$-function ($\phi^*$-function).


\begin{prop}
With $\phi(z)$ defined in (\ref{phi_def}), the following connection formulas
between the $\phi$-function and the $\wt\phi$-function ($\phi^*$-function)
hold:
\begin{align}
	\pt(z)=\phi(z)\pm \pi i (1-\frac1c z)
 \label{pt_phi}
\end{align}
and
\begin{align}
	\ps(z)=\phi(z)\pm \pi i \frac1c (1-z)
 \label{ps_phi}
\end{align}
for $z\in\mathbb{C}_\pm$.
\label{pro_pt_phi}
\end{prop}

\medskip
\begin{figure}[!htpb]
\begin{minipage}[t]{0.45\linewidth}
\centering
\includegraphics[scale=1]{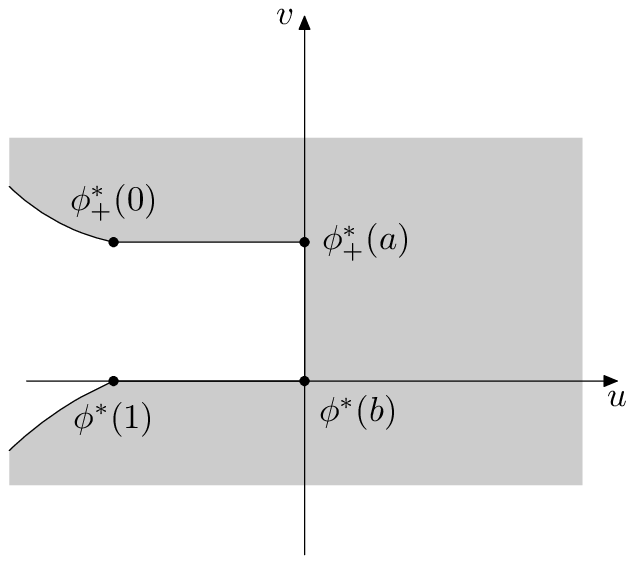}
\\
\medskip
(a) The upper half $z$-plane under $\ps(z)$
\end{minipage}
\hfill
\begin{minipage}[t]{0.45\linewidth}
\centering
\includegraphics[scale=1]{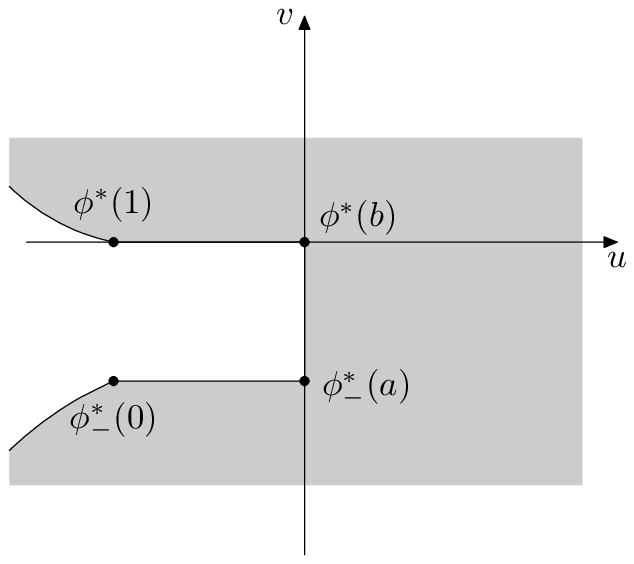}
\\
\medskip
(b) The lower half $z$-plane under $\ps(z)$
\end{minipage}
\caption{The image fo the $z$-plane  under the mapping  $\ps(z)$.}
\label{ps_fig}
\end{figure}

To establish \eqref{pt_phi} and \eqref{ps_phi}, we need some properties of $g'(z)$. In fact, we shall  derive  explicit formulas for $g(z)$ and $g'(z)$.  From (\ref{mu}), we have
\begin{align}
	g(z)
 =\frac1c
	\int_0^a \log(z-s)ds +\frac1c \int_b^1 \log(z-s)ds
	+\int_a^b\log(z-s) \mu(s)ds.
	\label{g_cal}
\end{align}
Integration by parts gives
\begin{align}
	\int_a^b  \log(z-s) \mu(s) ds
 &=
			 \mu(s)\log(z-s) s \Big |_{a}^{b}
    			-\int_a^b s [ \log(z-s) \mu(s) ]' ds
 \hh
        		 \mu(s)\log(z-s) s \Big |_{a}^{b}
        		-
        		\int_a^b \mu(s) ds
		-\int_a^b s  \log(z-s) \mu(s) ' ds
		+z\int_a^b \frac{1}{z-s} \mu(s) ds
\label{g_cal1}
.
\end{align}
The last integral in (\ref{g_cal1}) is given by
\begin{align}
	\int_a^b \frac{1}{z-s} \mu(s) ds
 =
 	-\mu(s)\log(z-s)  \Big |_{a}^{b}
    		+
       		\int_a^b  \mu(s) ' \log(z-s) ds
 \label{g_cal2}
 .
\end{align}
On account of (\ref{mu}) and (\ref{ab}), a direct calculation gives
\begin{align}
	\int_a^b \mu(s) ds
	=
	1-\frac{2a}{c}
 \label{mu_ab}
\end{align}
and
\begin{align}
	\mu(x)'
	=
	-\frac {1}{2\pi} \frac{1-2x}
	{ (x-x^2) \sqrt{ (x-a)(b-x) } }
,
\quad \qquad  x\in(a,b).
 \label{mu_der}
\end{align}
Applying (\ref{mu_der}) to (\ref{g_cal2}), we obtain
\begin{align}
	\int_a^b \frac{1}{z-s} \mu(s) ds
 &=-\frac1c \log(z-b)
  	+\frac1c \log(z-a)
  -\frac{  1  }{ 2\pi   } \int_a^b   \frac{\log(z-s)}{ s \sqrt{ (s-a)(b-s) } }ds
 \hhh
		+\frac{  1  }{ 2\pi   }  \int_a^b   \frac{\log(z-s)}
		{ (1-s) \sqrt{ (s-a)(b-s) } }ds.
 \label{g_cal3}
\end{align}
Similarly, one can show that
\begin{align}
	\int_a^b s  \log(z-s) \mu(s) ' ds
 =
 	-\frac {1}{\pi} \int_a^b   \frac{\log(z-s)}
		{  \sqrt{ (s-a)(b-s) } }ds
		+\frac {1}{2\pi} \int_a^b   \frac{\log(z-s)}
		{ (1-s) \sqrt{ (s-a)(b-s) } }ds.
 \label{g_cal4}
\end{align}
A combination of  \r{g_cal1},  (\ref{mu_ab}),  (\ref{g_cal3}) and \r{g_cal4}  gives
\begin{align}
	\int_a^b  \log(z-s) \mu(s) ds
	&=
	\bigg[
	\frac1c (z-a)\log(z-a)-
 	\frac1c (z-b)\log(z-b)
	\bigg]
	-(1-\frac{2a}{c})
        		\notag
        		\\
        		&\quad
        		-\frac{  z  }{ 2\pi   } \int_a^b   \frac{\log(z-s)}
		{ s \sqrt{ (s-a)(b-s) } }ds
        		+\frac{  1  }{ 2\pi   } (z-1) \int_a^b   \frac{\log(z-s)}
		{ (1-s) \sqrt{ (s-a)(b-s) } }ds
        		\notag
        		\\
        		&\quad
        		+
        		\frac {1}{\pi} \int_a^b   \frac{\log(z-s)}
		{  \sqrt{ (s-a)(b-s) } }ds
.
 \label{g_cal_osc}
\end{align}
Let the integrals on the right-hand side of the equality be denoted by $I_1$,
$I_2$ and $I_3$, respectively. To evaluate $I_1$,  we note that from \cite{WongWong} we have
\begin{align}
	\int_a^b \frac{  \log s   }{  (s-z) \sqrt{ (s-a)(b-s) } }ds
	=\frac{  2\pi  }{  (z-a)^{1/2}(z-b)^{1/2}  }
	\log \frac{  z+\sqrt{ ab }+(z-a)^{1/2}(z-b)^{1/2}  }
	{ (\sqrt a +\sqrt b)z   }
 .
\label{wangintegral}
\end{align}
Making the change of variable $s=z-x$,  we obtain from (\ref{wangintegral})
\begin{align}
		I_1
	&=
		\int_a^b   \frac{\log(z-s)}
		{ s \sqrt{ (s-a)(b-s) } }ds
	\hh
		\int_{z-a}^{z-b}   \frac{\log x }
		{ (z-x)  \sqrt{ (z-x-a)(b-z+x) } }d(z-x)
	\hh
	-\frac{4\pi}{c} \log
		\frac{ z+\sqrt{ (z-a)(z-b) }+c/2  }
		{ (\sqrt{z-a}+\sqrt{z-b})z }
 \label{I1}
\end{align}
and
\begin{align}
	I_2
	&=
 \int_a^b   \frac{\log(z-s)}
		{ (1-s) \sqrt{ (s-a)(b-s) } }ds
	\hh
		-\frac{4\pi}{c} \log
		\frac{ z-1+\sqrt{ (z-a)(z-b) }-c/2  }
		{ (\sqrt{z-a}+\sqrt{z-b})(z-1) }.
	\label{I2}
\end{align}
Also, note that
\begin{align}
	\int_a^b \frac{  \log s   }{   \sqrt{ (s-a)(b-s) } }ds
	=2\pi \log \frac{   \sqrt{a}+\sqrt{b} }{ 2   }
;
\label{t17}
\end{align}
see \cite[ Lemma 1]{WongWong}. Hence,
\begin{align}
		I_3
	&=\int_a^b   \frac{\log(z-s)}
		{  \sqrt{ (s-a)(b-s) } }ds
	=
		2\pi \log\frac{ \sqrt{z-b}+\sqrt{z-a} }{2}
\label{I3}
 .
\end{align}
A straightforward calculation shows that
\begin{align}
	g(z)
&=
	-1-2\log 2
           +\frac1c (z-1)\log(z-1)-\frac1c z\log z
                         +(2-\frac2c)\log\left[ (z-a)^{1/2}+ ({z-b})^{1/2} \right]
            \notag
            \\
            &\ +
                \frac2c z
                   \log\left[ z+  (z-a)^{1/2}(z-b)^{1/2}+c/2  \right]+\frac2c (1-z)
                \log
        		\left[ z-1+ (z-a)^{1/2}(z-b)^{1/2}-c/2  \right]
.
\label{g}
\end{align}
Moreover, by using   \eqref{g_cal3}, \eqref{I1} and \eqref{I2},   we have
\begin{align}
	g'(z)&=\int_a^b \frac{\mu(s)}{z-s} ds
	-\frac1c \log(z-a)+\frac1c \log z-\frac1c \log(z-1)+\frac1c \log(z-b)
\hh
\frac2c \log\(
	\frac{  z+\sqrt{(z-a)(z-b)}+c/2  }{   (z-1)+\sqrt{(z-a)(z-b)}-c/2 }
\)  +\frac1c \log\(  \frac{z-1}{z} \).
	\label{g:der}
\end{align}
\\
\begin{proof}[Proof of Proposition \ref{pro_pt_phi}  ] Note that when $x\in(a,b)$, $\nu_\pm(x)=\pm \pi i (\mu(x)-\frac1c)$; see \r{nu_pm}. Hence, it follows that
\begin{align}
	\nu_+(x)-\nu_-(x)=2\pi i (\mu(x)-\frac1c)
	.
 \label{nu_jum1}
\end{align}
On the other hand, if $x\in(-\infty,0) \cup (1,+\infty)$,
we have with the aid of (\ref{nu_pm1}) and (\ref{nu_pm2})
\begin{align}
	\nu_+(x)-\nu_-(x)=-\frac2c \pi i
	.
 \label{nu_jum2}
\end{align}

To establish (\ref{pt_phi}), we need one more function $\hat \phi(z)$
defined by
\begin{align}
	\hat \phi(z):= -g'(z) \mp \frac1c \pi i,
	\qquad z\in \cpm.
 \label{phat_def}
\end{align}
From  \eqref{g_def},  one can show that  for $x\in(0,1)$
\begin{align}
	g_+'(x)-g_-'(x)=-2\pi i \mu(x).
 \label{g_der_plu}
\end{align}
From (\ref{g_der_plu}), it is easily verified that $\hat \phi(z)$ is analytic for $z\in\mathbb{C}\setminus (-\infty,0]\cup [a,b]\cup[1,+\infty)$ and satisfies the jump conditions
\begin{align}
	\hat \phi_+(x)-\hat \phi_-(x)
	=
	2\pi i (\mu(x)-\frac1c),
	\qquad x\in(a,b)
 \label{phat_jum1}
\end{align}
and
\begin{align}
	\hat \phi_+(x)-\hat \phi_-(x)
	=-\frac2c \pi i,
	\qquad x\in(-\infty,0) \cup (1,+\infty).
 \label{phat_jum2}
\end{align}
Comparing (\ref{phat_jum1}) and (\ref{phat_jum2}) with (\ref{nu_jum1})
and (\ref{nu_jum2}), we find
\begin{align}
	\nu_+(x)-\nu_-(x)=\hat \phi_+(x) -\hat \phi_-(x)
\end{align}
or, equivalently,
\begin{align}
	\nu_+(x)-\hat \phi_+(x)=\nu_-(x) -\hat \phi_-(x)
 \label{nu_ph}
\end{align}
for $x\in(-\infty,0)\cup(a,b)\cup(1,+\infty)$. On account of this,
$\nu(z)-\hat \phi(z)$ can be analytically extended to the whole complex plane.
Note that $\nu_+(x)=-\nu_-(x)$,  and  that from \eqref{g:der} it can be shown
\begin{align}
	g_+'(x)+g_-'(x)
&=\frac2c \log\( \frac{ x+xc }{  (1-x)(1+c) } \)
	+\frac2c \log \(  \frac{1-x}{x} \)
=
	0
 \qquad x\in(a,b)
\label{g:der:1}
\end{align}
and so
\begin{align}
	\hat \phi_+(x)=-\hat \phi_-(x),
\qquad x\in(a,b).
\label{ph:1}
\end{align}
Hence, from \r{nu_ph}, \r{nu_pm} and \r{ph:1} we have
\begin{align}
	\nu(x)-\hat \phi(x)
	&=\frac12\big[ (\nu-\hat\phi)_+(x)+(\nu-\hat\phi)_-(x)  \big]
 \hh
 	0,
 	\qquad x\in(a,b).
\label{nu:phih}
\end{align}
Using analytic continuation, it follows from (\ref{phat_def})
\begin{align}
	\nu(z)=-g'(z) \mp \frac1c \pi i,
	\qquad z\in\mathbb{C}_\pm.
 \label{nu_g_der}
\end{align}
Hence, it follows from  (\ref{pt_def}) that
\begin{align}
	\pt(z)
	=\int_a^z \bigg( -g'(s) \mp \frac1c \pi i \bigg) ds,
	\qquad z\in\mathbb{C}_\pm.
\end{align}
Now, let $z\in\mathbb{C}_+$;  on account of  \r{g_def} and the definition of $l$, we have
\begin{align}
	\pt(z)
 &=
	\int_a^z \bigg(- g'(s) - \frac1c \pi i \bigg) ds
 \hh
 	-g(z)+g_+(a)-\frac1c \pi i (z-a)
 \hh
 	l/2-g(z)+\pi i(1-\frac1c z ).
 \label{pt_ps_pro1}
\end{align}
In a similar manner, one can show that for $z\in\mathbb{C}_-$
\begin{align}
	\pt(z)=
 	l/2-g(z)- \pi i (1-\frac1c z).
 \label{pt_ps_pro2}
\end{align}
Formula
(\ref{pt_phi}) now  follows from (\ref{pt_ps_pro1}), (\ref{pt_ps_pro2}) and (\ref{phi_def}).
Formula (\ref{ps_phi}) is obtained by coupling   (\ref{pt_phi}) and (\ref{ps_pt}).
\end{proof}

We conclude this section with the mapping properties of $\phi(z)$ in (\ref{phi_def}). By using Proposition \ref{pro_pt_map} and (\ref{pt_phi}),
one can prove the following proposition; see  also Figure \ref{phi_fig}. (To establish the second equality in \r{phi_map_ab}, we make use of \r{ps_phi} instead of \r{pt_phi}.)

\begin{prop}
The mapping properties of $\phi(z)$ are given by:
\begin{align}
	\Re \phi_\pm (x) <  \Re \phi_\pm(0), \quad
	\Im \phi_\pm (x)=\mp\pi , \quad
	x\in(-\infty,0);
\label{phi_map_min}
\\
	\Re \phi_\pm(x)<0, \quad
	\Im \phi_\pm(x)= \mp(1-\frac1c x) \pi, \quad
	x\in[0,a);
\label{phi_map_0}
\\
	\phi_\pm (a)=\mp(1-\frac1c a) \pi i, \quad
	 \phi_\pm(b)=\mp \frac1c (1-b) \pi i ;
\label{phi_map_con}
\\
	\Re \phi_\pm(x)=0, \quad
	\arg \phi_\pm(x)=\mp \frac\pi 2, \quad
	x\in(a,b);
\label{phi_map_ab}
\\
	\Re \phi_\pm(x)<0, \quad
	\Im \phi_\pm(x)=\mp \frac1c(1-x)\pi, \quad
	x\in(b,1);
\label{phi_map_b}
\\
	\phi (x) <  \phi(1)<0, \quad
	x\in(1,+\infty).
\label{phi_map_1}
\end{align}
\label{pro_phi_map}
\end{prop}

\begin{figure}[!htpb]
\begin{minipage}[t]{0.45\linewidth}
\centering
\includegraphics[scale=1]{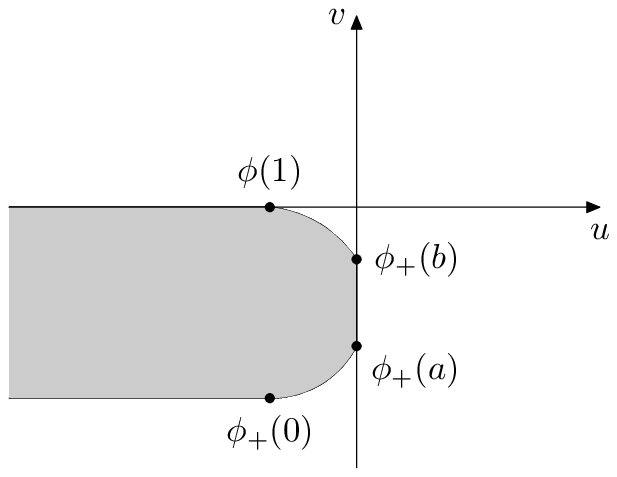}
\\
\medskip
(a) The upper half $z$-plane under $\phi(z)$
\end{minipage}
\hfill
\begin{minipage}[t]{0.45\linewidth}
\centering
\includegraphics[scale=1]{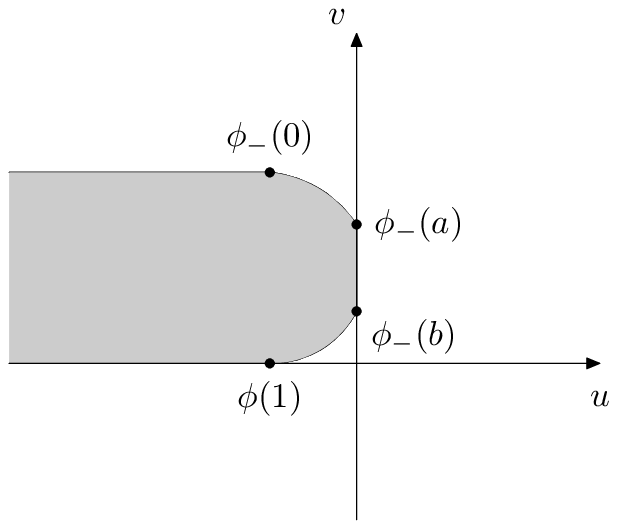}
\\
\medskip
(b) The lower half $z$-plane under $\phi(z)$
\end{minipage}
\caption{The  image of the $z$-plane  under the mapping  $\phi(z)$.}
\label{phi_fig}
\end{figure}

\section{Normalization and Contour Deformation}

Recall from \r{T} that
\begin{align}
  T(z)=(c_1e^{ nl })^{-\s_3/2} R(z)
	e^{-n\wt g(z)\s_3}
	  (c_1e^{ nl })^{\s_3/2},
\end{align}
where $l$ is the constant given in Definition 4.1 and $\wt g(z)$ is shown in \eqref{gt_def}. By using \r{R_real}, \r{R:jum:S}, \r{g_z} and \eqref{R_inf},  a  straightforward calculation shows that $T(z)$ is the unique solution of
the following RHP:
\begin{enumerate}
\item[($T_{a}$)]
$T(z)$ is analytic for $z\in \mathbb{C}\setminus\Sigma$;
\item[($T_{b}$)] for $z\in \Sigma_\pm$,
\begin{equation}
	 T_+(z)
=
    T_-(z)
    \(
    \begin{aligned}
    &\qquad\quad  1  && 0 \\
    \medskip
    &c_1\wt r_\pm(z)e^{n(l-2\gt(z))  } && 1
    \end{aligned}
    \)
    ,
 \label{T:jum1}
\end{equation}
and for $z\in(0,1)$,
\begin{equation}
    T_+(z)
=
    T_-(x)
    \begin{pmatrix}
    J_{11}(x) & J_{12}(x)\\
    J_{21}(x) & J_{22}(x)
    \end{pmatrix},
  \label{T:jum2}
\end{equation}
where
\begin{align*}
	J_{11}(x)&=e^{-n(\gt_+(x) -\gt_-(x)  )},
\\
	J_{22}(x)&=e^{n(\gt_+(x) -\gt_-(x)  )},
\\
	J_{21}(x)&=c_1r(x)e^{n(l-\gt_-(x)-\gt_+(x))  },
\end{align*}
and
\begin{align*}
	J_{12}(x)=0;
\end{align*}
\item[($T_{c}$)] $T(z)$ behaves like the identity matrix at infinity:
\begin{align*}
	T(z)=I+O(1/z)
	\qquad\qquad \mbox{as } z\ra \infty.
\end{align*}
\end{enumerate}
For the sake of simplicity, we introduce some auxiliary functions. Let
\begin{align}
	E(z)&:= \exp\(-N\int_0^1 \log(z-s) ds \)
	\prod_{j=0}^{N-1}(z-x_{N,j})
	\hh
	e^{-N\varphi(z) }
	\prod_{j=0}^{N-1}(z-x_{N,j})
 \label{E}
\end{align}
for $z\in \mathbb{C}\setminus[0,1]$ and
\begin{align}
	\wt E(z):=\frac{   E(z) e^{ \mp N\pi i z }  }{  2\cos(N\pi z)  }
 \label{Et}
\end{align}
for $z\in\mathbb{C}_\pm$, where $\varphi(z)$ is defined in \r{vp}. In view of the equation preceding  (4.5), it is easily verified that
\begin{align}
	E_+(x)/E_-(x)=e^{ N(\varphi_-(x)-\varphi_+(x)) }
	=e^{2N\pi ix},
	\label{E_pm}
	\qquad x\in(0,1).
\end{align}
On account of this,  it can be shown that $\et_+(x)\et_-(x)^{-1}=1$ for $x\in(0,1)$. Hence,
$\et(z)$ can be  analytically continued  to the interval  $(0,1)$.  Moreover,
\begin{align}
	\et (x)^2 = \frac{ E_+(x)E_-(x) }{ 4\cos^2(N\pi x) },
	\qquad\qquad x\in(0,1).
	\label{Et2}
\end{align}
As $N \rightarrow \infty$, by applying Stirling\rq{}s formula to
\eqref{E} and evaluating the integral in \r{vp} explicitly, one can show that
$
	E(z)\sim 1
$
uniformly for $z$ bounded away from the interval $[0,1]$.
Moreover, from \r{Et}, we have for $z\in\mathbb{C}_\pm$
\begin{align}
	E(z)/\wt E(z)=1+e^{\pm 2N\pi iz} \sim 1
 \label{E_asy}
\end{align}
uniformly for $z$ bounded away from the real line.

To compute the jump conditions of $T(z)$ in ($T_b$) above, we first state the following lemma.
\begin{lemma}
For $0<x<1$, we have
\begin{align}
 {r(x)}{  e^{ n (l-\wt g_+(x) - \wt g_-(x)) } } =\frac{
e^{n( \ps_+(x) +\ps_-(x)  ) }   }{ c_1 w(x)  \widetilde{E}(x)^2 },
 \label{lem_E1}
\end{align}
and
\begin{align}
	e^{n(\gt_+(x)-\gt_-(x)) }=\left\{
	\begin{array}{ll}
	1, \qquad & x\in(0,a)\cup (b,1),\\
	e^{2n\ps_-(x)  }=e^{-2n\ps_+(x) },\qquad  &  x\in(a,b).
	\end{array}
	\right.
 \label{gt:e}
\end{align}
For $z\in\Sigma_{\pm}$, we have
\begin{align}
{\widetilde{r}_\pm(z)}{e^{
n ( l-2\wt g(z) ) } } = \frac{ e^{ 2n\phi(z) }  }{ \mp c_1 w(z) \widetilde{E}(z) E(z) }.
 \label{lem_E2}
\end{align}
\label{lem_gt}
\end{lemma}
\begin{proof}
	For  $0<x<1$, it follows from  \r{r} and \r{gt_def} that
\begin{align}
 {r(x)}{  e^{ n (l-\wt g_+(x) - \wt g_-(x)) } }
&=\frac{4\cos^2(N\pi x )}{ c_1 w
e^{ -N(\varphi_+ +\varphi_-) } \prod_{j=0}^{N-1}(x-x_{N,j})^{2}}
e^{ n(l-g_+ -g_- ) }
\hh
	\frac{4\cos^2(N\pi x )}{ c_1 w
	E_+ E_- }e^{n( \ps_+ +\ps_- )}.
  \label{lem_t1}
\end{align}
The second equality is obtained by using \r{phi_def},  \r{ps_phi} and \r{E}.
Substituting \eqref{Et2} into \eqref{lem_t1} gives \eqref{lem_E1}. For $x\in(0,a)\cup(b,1)$, the formula $e^{n(\gt_+(x)-\gt_-(x)) }=1$ follows immediately from \eqref{g_jum2}. Note that $\nu_\pm(x)=\pm\pi i(\mu(x)-\frac1c)$ for $x\in(a,b)$;  hence, by \eqref{g_jum2} and \eqref{ps_def}  we obtain
\begin{align}
	e^{n(\gt_+(x)-\gt_-(x)) }
	=e^{ 2n \int_b^x \nu_-(s) ds }=
	e^{ 2n\ps_-(x) }=e^{-2n\ps_+(x)},
	\qquad x\in(a,b).
 \label{lem_t2}
\end{align}
For $z\in\Sigma_\pm$,
\eqref{lem_E2} follows directly from \eqref{phi_def}, \eqref{E} and \eqref{Et}.

\end{proof}
With the properties of  $\pt(z)$ and $\ps(z)$ established in Section 4
and Lemma \ref{lem_gt}, the jump matrices for $T$ in condition ($T_b$) can be written as follows:
\begin{equation}
T_+(x)=T_-(x)
\begin{pmatrix}
1 & 0\\
 {e^{2n\widetilde{\phi} }   }/({ w  \widetilde{E}^2 })  & 1
\end{pmatrix}\label{T_jum0}
\end{equation}
for $0<x<a$,
\begin{equation}
T_+(x)=T_-(x)
\begin{pmatrix}
e^{-2n\ps_- } & 0\\
{1 }/({ w  \widetilde{E}^2 })  & e^{-2n\ps_+ }
\end{pmatrix}\label{T_juma}
\end{equation}
for $a<x<b$,
\begin{equation}
T_+(x)=T_-(x)
\begin{pmatrix}
1 & 0\\
{e^{2n{\phi}^* } }/({ w  \widetilde{E}^2 })  & 1
\end{pmatrix}\label{T_jumb}
\end{equation}
for $b<x<1$, and
\begin{equation}
T_+(z)=T_-(z)
\begin{pmatrix}
1 & 0\\
{\mp e^{2n\phi }   }/({  w  E\widetilde{E} })  & 1
\end{pmatrix}\label{T_jums}
\end{equation}
for $z\in \Sigma_{\pm}$.

Note that for the jump matrix on $(a,b)$, we have the following factorization:
\begin{equation}
\begin{pmatrix}
e^{-2n\ps_-} & 0 \\\\  {1}/({w\et^2})  &   e^{-2n\ps_+}
\end{pmatrix}
=\begin{pmatrix}
\et & { w\et }/{e^{2n\ps_-} } \\\\
0 & 1/\et
\end{pmatrix}
\begin{pmatrix}
0 & -w \\
w^{-1} & 0
\end{pmatrix}
\begin{pmatrix}
1/\et & { w\et }/{e^{2n\ps_+} } \\\\
0 & \et
\end{pmatrix}.  \label{T_fac}
\end{equation}
The first and third matrices on the right-hand side of \eqref{T_fac} have,  respectively, the
analytic continuations
$
	\begin{pmatrix}
\et & { w\et }/{e^{2n\ps } } \\
0 & 1/\et
\end{pmatrix}
$
and
$
	\begin{pmatrix}
1/\et & { w\et }/{e^{2n\ps } } \\
0 & \et
\end{pmatrix}
$
on both sides of $(a,b)$.
Let $\Sigma_\pm=\bigcup_{i=1}^3 \Sigma^i_\pm$, and let
$\Sigma_V=\bigcup_{i=1}^3 \Sigma^i_\pm\cup \Sigma_\pm^{a}
\cup \Sigma_\pm^{b}
$
denote the contour shown in Figure \ref{T_fig}. Define
\begin{align}
	&V(z):=T(z)E(z)^{\s_3}, \qquad  \qquad\qquad \qquad \quad \  z\in \Omega_\infty;
 \label{V_T1}
\\
	&V(z):=T(z)\wt E(z)^{\s_3}, \qquad \qquad\qquad \qquad \quad \
	z\in \Omega_\pm^1\cup \Omega_\pm^3;
 \label{V_T2}
\\
	&V(z):=T(z)
	\begin{pmatrix}
\wt {E} &   { -w\wt {E} }/{e^{2n\ps } } \\
0 & 1/\wt {E}
\end{pmatrix},
\qquad
\quad\ \  z\in \Omega_+^2;
  \label{V_T3}
\\
&V(z):=T(z)
	\begin{pmatrix}
\wt {E} &  { w\wt {E} }/{e^{2n\ps } } \\
0 & 1/\wt {E}
\end{pmatrix},
\qquad
\qquad\  z\in \Omega_-^2,
  \label{V_T4}
\end{align}
		\begin{figure}[!b]
		\centering
		\includegraphics[scale=0.97]{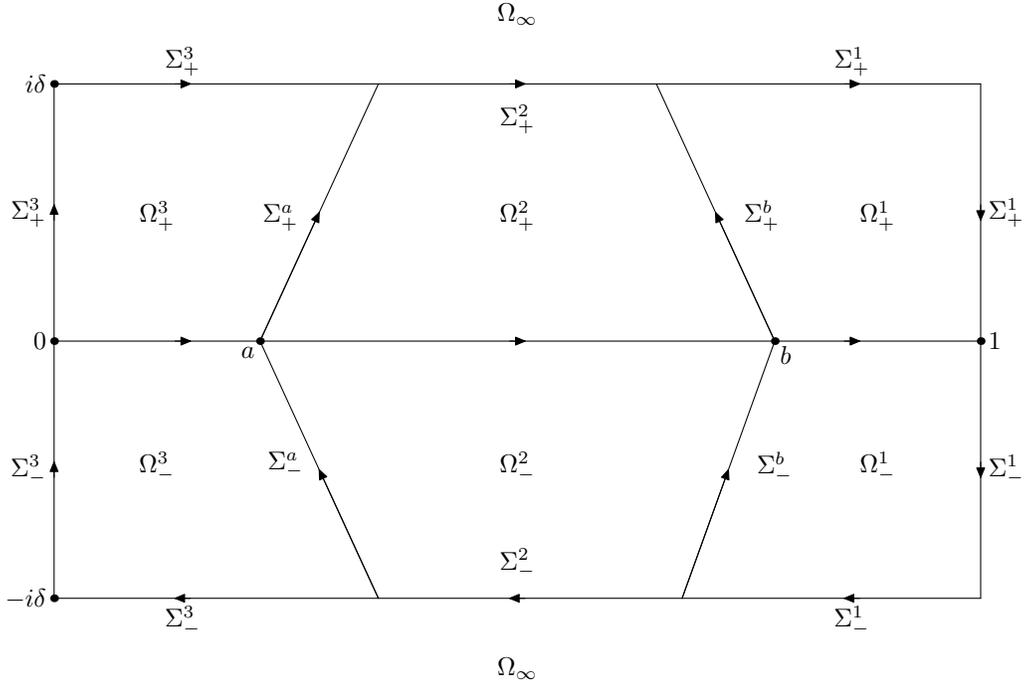}
		\caption{Contour $\Sigma_V$ and regions $\Omega_\pm^1, \Omega_\pm^2, \Omega_\pm^3, \Omega_\infty $. }
		\label{T_fig}
		\end{figure}where regions $\Omega_\pm=\bigcup_{i=1}^3 \Omega^i_\pm$ and $\Omega_\infty$ are also depicted in Figure \ref{T_fig}. It can be verified that the jump matrices of $V(z)$ are given by
\begin{align}
	J_V(x)&=
	\begin{pmatrix}
	0 & -w \\
	w^{-1} & 0
	\end{pmatrix},
 \qquad z\in(a,b);
 \label{JV:a}
 \\
 	J_V(x)&=
	\begin{pmatrix}
	1 & 0\\
	e^{2n\widetilde{\phi} }    w^{-1}    & 1
	\end{pmatrix},
 \qquad z\in(0,a);
 \label{JV:0}
 \\
 	J_V(x)&=
	\begin{pmatrix}
	1 & 0\\
	e^{2n {\phi}^* }  w^{-1}   & 1
	\end{pmatrix},
 \qquad z\in(b,1);
 \label{JV:b}
\end{align}
and
\begin{align}
J_V(x)&=
\begin{pmatrix}
E /\widetilde{E} & 0\\
	\mp w^{-1} e^{2n\phi }    &  \wt E/E
\end{pmatrix},
\qquad
z\in\Sigma_{\pm}^1 \cup \Sigma_{\pm}^3;
	\label{JV:S13}
\\
\notag
\\
J_V(z)&=
\begin{pmatrix}
1 & { \pm w\tilde{E} }/({e^{2n\ps}E })\\\\
\mp w^{-1} {e^{2n\phi }   } &  {\widetilde{E}}/{E}
\end{pmatrix},
\qquad
z\in\Sigma_{\pm}^2;
 \label{JV:S2}
\\
\notag
\\
J_V(z)
&=
\begin{pmatrix}
1 & { \pm w }/{e^{ 2n\ps} } \\
0 & 1
\end{pmatrix},
\qquad z\in\Sigma_{\pm}^a;
	\label{JV:Sa}
\\
\notag
\\
J_V(z)&=
\begin{pmatrix}
1 &  {\mp w }/{e^{ 2n\ps } } \\
0 & 1
\end{pmatrix},
\qquad z\in\Sigma_{\pm}^b.
 \label{JV:Sb}
\end{align}
In proving \r{JV:a}, \r{JV:0} and \r{JV:b}, we note that $\et$ is analytic in $0<x<1$. For \r{JV:a}, we also make use of \r{T_fac}. In establishing \r{JV:S2}, we have used \r{ps_phi} and \r{E_asy}. For
\r{JV:Sa} and \r{JV:Sb}, we note that $T_-(z)^{-1}T_+(z)=I$. Thus, $V(z)$ is a solution of the following RHP:
\begin{enumerate}
\item[($V_{a}$)]
$V(z)$ is analytic for $z\in \mathbb{C}\setminus \Sigma_V \cup [0,1]$;
\item[($V_b$)] for $z \in \Sigma_V$,
\begin{equation*}
    V_+(z)
=
    V_-(z)
    J_V(z);
\end{equation*}
\item[($V_c$)] as $z\rightarrow \infty$,
\begin{align*}
	V(z)=I+O(1/z).
\end{align*}
\end{enumerate}
For condition ($V_c$),  we note from \r{E} that $E(z) \sim 1$ as $z\ra \infty$.
From the properties of $\ps(z)$ and $\phi(z)$  (see  also Figure \ref{pt_fig} and Figure \ref{phi_fig}),  the  regions  $\Omega_\pm^2$ and $\Omega_\pm$ can be chosen sufficiently small so that
\begin{equation}
 \begin{array}{ll}
  \Re \ps(z) > 0
	& \mbox{for }
	z\in \Omega_\pm^2,
	\\
 	\Re \phi(z)<0
	& \mbox{for }
	z\in \Omega_\pm.
 \end{array}
 \label{phi_0}
\end{equation}
Moreover, from (\ref{E_asy}) it follows that $E/\et \sim 1$ as $n\ra\infty$ for $z$ bounded away from the real line.
These together with (\ref{pt_map_0}) and (\ref{ps_map_b}) imply
that the jump matrix $J_V (z)$ tends exponentially to the identity matrix as $n\rightarrow \infty$ for $z$ bounded away from $[a,b]\cup \{0\}\cup \{1\}$.
On the other hand,
the weight function $w(z)$ in (\ref{w_def}) can be rewritten as
\begin{align}
  w(z)
  &=\a! \b!N^{-\a-\b} \rho(Nz-1/2;\a,\b,N-1)
 \hh
  \frac{ \Gamma(Nz+\a+1/2)\Gamma( N(1-z)+\b+1/2  ) }{  N^{\a+\b}\Gamma(Nz+1/2)\Gamma(N(1-z)+1/2)    }.
 \label{w_rep}
\end{align}
A direct calculation shows that
\begin{align}
  w(z) \sim   z^{\a} (1-z)^{\b}
  \qquad \quad  \mbox{as } N\ra \infty.
 \label{w_asy}
\end{align}
Therefore, it is natural to suggest that for large $n$, the solution
of the RHP for $ V(z) $ may behave asymptotically like the solution of the following RHP for $N(z)$:
\begin{enumerate}
\item[($N_{a}$)]
$N(z)$ is analytic for $z\in \mathbb{C}\setminus[a,b]$;
\item[($N_b$)] for $x\in (a,b)$,
\begin{equation}
    N_+(x)
=
    N_-(x)
    \begin{pmatrix}
    0 & -h(x)\\
    h(x)^{-1} & 0
    \end{pmatrix},
  \label{N_jum}
\end{equation}
where
\begin{align}
	h(x)=x^{\a}(1-x)^{\b};
 \label{h}
\end{align}
\item[($N_c$)] as $z\rightarrow \infty$,
\begin{align*}
	N(z)=I+O(1/z).
\end{align*}
\end{enumerate}

To solve the above RHP, we first consider the function
\begin{align}
	M(z):=
	\(
	\frac{  z+c/2+ (z-a)^{1/2}(z-b)^{1/2}  }{  z( \sqrt a +\sqrt b )  }
	\)^{\a}
	\(
	\frac{  1-z+c/2-(z-a)^{1/2}(z-b)^{1/2}    }{   (1-z) ( \sqrt a +\sqrt b )  }
	\)^{\b},
	\label{M}
\end{align}
where $(z-a)^{1/2}(z-b)^{1/2} $ is analytic in $\mathbb{C}\setminus[a,b]$ and
behaves like $z$ as $z\ra \infty$. In view of \eqref{ab}, it is easily shown that the function $M(z)$ is indeed  analytic in $\mathbb{C}\setminus[a,b]$ and
\begin{align}
	M_+(x)M_-(x)= x^{-\a} (1-x)^{-\b}
	\qquad \mbox{for } x\in(a,b).
 \label{M_jum}
\end{align}
Its limit at infinity is given by
\begin{align}
	M_\infty=\lim_{z\rightarrow \infty}M(z)=
	\( 2/(  \sqrt a+\sqrt b )   	  \)^{\a+\b}.
 \label{M_inf}
\end{align}
Therefore, if we set
\begin{align}
	\wt N(z):=M_\infty^{\s_3} N(z)  M(z)^{-\s_3},
	\label{Nt_def}
\end{align}
it follows from  \eqref{N_jum}, \eqref{M_jum} and \eqref{M_inf}  that $\wt N(z)$ is a solution of the RHP:
\begin{enumerate}
\item[($\wt N_{a}$)]
$\wt N(z)$ is analytic for $z\in \mathbb{C}\setminus[a,b]$;
\item[($\wt N_b$)] for $x\in (a,b)$,
\begin{equation}
    \wt N_+(x)
=
   \wt N_-(x)
    \begin{pmatrix}
    0 & -1\\
    1 & 0
    \end{pmatrix};
  \label{Nt_jum}
\end{equation}
\item[($\wt N_c$)] as $z\rightarrow \infty$,
\begin{align*}
	\wt N(z)=I+O(1/z).
\end{align*}
\end{enumerate}
This problem can be solved explicitly, and its solution is given by
\begin{align}
  \wt N(z)=\frac12
  \begin{pmatrix}
	a(z)+a(z)^{-1} & i(a(z)-a(z)^{-1}) \\
	i(a(z)^{-1}-a(z))  &  a(z)+a(z)^{-1}
\end{pmatrix},
\label{Nt}
\end{align}
where
\begin{align}
	a(z)=\frac{(z-b)^{1/4}}{(z-a)^{1/4}}
\end{align}
with a branch cut along $[a,b]$ and $a(z)\rightarrow 1$ as $z\rightarrow \infty$.

Hence,
\begin{align}
	N(z)=\frac12 M_\infty^{-\s_3} \begin{pmatrix}
	a(z)+a(z)^{-1} & i(a(z)-a(z)^{-1}) \\
	i(a(z)^{-1}-a(z))  &  a(z)+a(z)^{-1}
\end{pmatrix}
  M(z)^{\s_3},
  \label{N}
\end{align}
and for $z\in\mathbb{C}\setminus\{0,1,a,b\}$, we expect that
\begin{align}
	V(z)\sim  N(z);
 \label{V_asy}
\end{align}
see the comment following equation \r{w_asy}. However, since $E/\et\not\sim 1 $ for $z$ near 0 and 1  and since $e^{-2n\ps(z)} \not \sim 0$ for $z$ near $a$ and $b$, the jump matrix $J_V(z) \not \sim I$ near these critical points; see \r{JV:S13}-\r{JV:Sb}. That is, $V(z)$ and $N(z)$ are not uniformly close to each other near these points, and special attention must be paid to the neighborhoods of these points.


Let $U(0,\varepsilon):=\{ z\in\mathbb{C} : |z|<\varepsilon \}$ be a neighborhood  of the origin, where $\varepsilon$ is a small positive number.
In view of \eqref{Et}, \eqref{JV:S13}, \eqref{phi_0} and \eqref{V_asy}, we first wish to find a matrix-valued function $V_0(z)$ that satisfies the following RHP:
\begin{enumerate}
\item[($V_{0,a}$)]
$V_0(z)$ is analytic for $z\in U(0,\varepsilon)\setminus(-i\varepsilon,i\varepsilon)$;
\item[($V_{0,b}$)]
\begin{equation}
    (V_{0})_+(z)=(V_{0})_-(z)
    \left\{
    \begin{array}{ll}
    	\begin{pmatrix}
		1+e^{2N\pi iz } & 0 \\
		0 & \(1+e^{2N\pi iz }\)^{-1}
	\end{pmatrix} &  \mbox{for } z\in (0,+i\varepsilon), \\\\
	\begin{pmatrix}
		1+e^{-2N\pi iz } & 0 \\
		0 & (1+e^{-2N\pi iz })^{-1}
	\end{pmatrix} & \mbox{for } z\in (-i\varepsilon,0),
	\end{array}
	\right.
 \label{V0_jum}
\end{equation}
where the functions $(V_0)_+(z)$ and $(V_0)_-(z)$ denote the boundary values of $V_0(z)$ taken from the left and right of the imaginary axis, respectively;
\item[($V_{0,c}$)]for $z\in \partial U(0,\varepsilon)$,
\begin{align}
	V_0(z)\sim N(z)
 \label{V0_asy}
\end{align}
as $n\rightarrow \infty$.
\end{enumerate}

Next, we consider as in \cite{WangWong} the following $D$-function
\begin{align}
	D(z):=\frac{e^{Nz}\Gamma(Nz+1/2) }{ \sqrt{2\pi }  (Nz)^{Nz}  },
	\qquad \qquad
	z \in \mathbb{C}\setminus (-\infty,0\,].
 \label{D}
\end{align}
By applying Stirling\rq{}s formula to \eqref{D}, we have
\begin{align}
	D(z)=1+O(1/n)
	\qquad \quad \mbox{as } n\ra \infty,
 \label{D_asy}
\end{align}
for $|\arg z|<\pi$. We also introduce the function
\begin{align}
 D^*(z) := \left\{
 	\begin{array}{ll}
		(1+e^{2N\pi i z}) D(z), & z\in\mathbb{C}_+, \\
		(1+e^{-2N\pi i z}) D(z), & z\in\mathbb{C}_-,
	\end{array}
	\right.
 \label{Ds_D}
\end{align}
which is actually analytic in $(-\infty,0\,)$.
From  Euler\rq{}s reflection formula, it follows that
\begin{align}
	D^*(z)=\frac{ \sqrt{2\pi } e^{Nz} (-Nz)^{ -Nz }  }{ \Gamma(-Nz+1/2) },
	\qquad \qquad
	z\in \mathbb{C}\setminus [\,0,+\infty)
 \label{Ds}
 .
\end{align}
Similarly,   for $|\arg (-z)|<\pi $, we have
\begin{align}
	D^*(z) =1 + O(1/n)
 \label{Ds_asy}
\end{align}
as $n\rightarrow \infty$.

Set the following matrix-valued function
\begin{align}
	Q_0(z)= \left\{
	\begin{array}{ll}
	\begin{pmatrix}
	D(z)  &  0  \\
	0  &  D(z)^{-1}
	\end{pmatrix}=
	D(z)^{\s_3}, & \Re z>0, \\\\
	\begin{pmatrix}
	D^*(z)  &  0  \\
	0  &  D^*(z)^{-1}
	\end{pmatrix}=
	D^*(z)^{\s_3}, & \Re z<0.
	\end{array}
	\right.
\end{align}
It follows from \eqref{D_asy}, \eqref{Ds_D} and \eqref{Ds_asy} that $Q_0$ is a solution of the RHP:
\begin{enumerate}
\item[($Q_{0,a}$)]
$Q_0(z)$ is analytic for $z\in \mathbb{C}\setminus(-i\infty,i\infty)$;
\item[($Q_{0,b}$)]
\begin{equation}
    (Q_{0})_+(z)=(Q_{0})_-(z)
    \left\{
    \begin{array}{ll}
    	\begin{pmatrix}
		1+e^{2N\pi iz } & 0 \\
		0 & (1+e^{2N\pi iz })^{-1}
	\end{pmatrix} &  \mbox{for } z\in(0,+i\infty ), \\\\
	\begin{pmatrix}
		1+e^{-2N\pi iz } & 0 \\
		0 & (1+e^{-2N\pi iz })^{-1}
	\end{pmatrix} & \mbox{for } z\in(-i\infty,0),
	\end{array}
	\right.
 \label{Q0_jum}
\end{equation}
where the functions $(Q_0)_+(z)$ and $(Q_0)_-(z)$ denote the boundary values of $Q_0(z)$ taken from the left and right of the imaginary axis, respectively;
\item[($Q_{0,c}$)] for $z$ bounded away from the origin,
\begin{align}
	Q_0(z)\sim I
 \label{Q0_asy}
\end{align}
as $n\ra \infty$.
\end{enumerate}
A comparison of conditions $(Q_{0,b})$ and $(Q_{0,c})$ with \eqref{V0_jum} and
\eqref{V0_asy} shows that a solution of the RHP for $V_0$ is given by
\begin{align}
 V_0(z)=N(z)Q_0(z)= \left\{
    \begin{array}{ll}
      N(z) D(z)^{\s_3}, & \Re z>0, \\
      N(z) D^*(z)^{\s_3}, & \Re z<0.
    \end{array}
  \right.
 \label{V0}
\end{align}

In a similar manner, for $z\in U(1,\varepsilon)=\{ z\in\mathbb{C} : |z-1|<\varepsilon \}$, we want to find a matrix-valued function $V_1(z)$ that satisfies the RHP:
\begin{enumerate}
\item[($V_{1,a}$)]
$V_1(z)$ is analytic for $z\in U(1,\varepsilon)\setminus(1-i\varepsilon,1+i\varepsilon)$;
\item[($V_{1,b}$)]
\begin{equation}
    (V_{1})_+(z)=(V_{1})_-(z)
    \left\{
    \begin{array}{ll}
    	\begin{pmatrix}
		1+e^{2N\pi iz } & 0 \\
		0 & (1+e^{2N\pi iz })^{-1}
	\end{pmatrix} &  \mbox{for } z\in  (1,1+i\varepsilon), \\\\
	\begin{pmatrix}
		1+e^{-2N\pi iz } & 0 \\
		0 & (1+e^{-2N\pi iz })^{-1}
	\end{pmatrix} & \mbox{for } z\in  (1-i\varepsilon,1),
	\end{array}
	\right.
 \label{V1_jum}
\end{equation}
where the functions $(V_1)_+(z)$ and $(V_1)_-(z)$ denote the boundary values of $V_1(z)$ taken from the right and left of the line $\Re z=1$, respectively; also see Figure \ref{T_fig};
\item[($V_{1,c}$)]for $z\in \partial U(1,\varepsilon)$,
\begin{align}
	V_1(z)\sim N(z)
 \label{V1_asy}
\end{align}
as $n\rightarrow \infty$.
\end{enumerate}
As in the case of $V_0$, it can be verified that a solution of the RHP for $V_1$ is given by
\begin{align}
 V_1(z)=\left\{
    \begin{array}{ll}
      N(z) D^*(1-z)^{\s_3}, & \Re z>1, \\
      N(z) D(1-z)^{\s_3}, & \Re z<1.
    \end{array}
  \right.
 \label{V1}
\end{align}

Since $V_0(z)$ satisfies the same jump conditions as $V(z)$ on the contour $\partial U(0,\varepsilon)\cup (0,+i\varepsilon) \cup (-i\varepsilon,0)$, $V(z)\sim V_0(z)$ for $z\in U(0,\varepsilon)$. Similarly, $V(z)\sim  V_1(z)$ for $z\in U(1,\varepsilon)$.
Hence, as $n\ra \infty$, we have from \eqref{V_asy}
\begin{align}
  V(z)\sim \left\{
  \begin{array}{ll}
   N(z),  & z\in (\Omega_\infty \cup \Omega_\pm^1 \cup \Omega_\pm^3) \setminus \( U(0,\varepsilon)\cup
 U(1,\varepsilon) \), \\
  V_0(z), & z\in U(0,\varepsilon),\\
  V_1(z), & z\in U(1,\varepsilon).
  \end{array}
  \right.
\label{V_asyr}
\end{align}
By \eqref{V_T1}-\eqref{V_T2}, it is clear that
\begin{align}
  T(z)= \left\{
  \begin{array}{ll}
   V(z)E(z)^{-\s_3},  & z\in \Omega_\infty, \\
    V(z)\et(z)^{-\s_3}, & z\in \Omega_\pm^1\cup \Omega_\pm^3.
  \end{array}
  \right.
 \label{T_V}
\end{align}

Let $x_0$ be an arbitrary fixed point in the interval $(a,b)$. From
\eqref{V0}, \eqref{V_asyr} and \eqref{T_V}, it follows that for $\Re z < x_0$,
\begin{align}
  T(z)\sim \left\{
  \begin{array}{ll}
   N(z)E(z)^{-\s_3},  & z\in \Omega_\infty\setminus U(0,\varepsilon), \\
  N(z) \et(z)^{-\s_3}, & z\in \Omega_\pm^3 \setminus U(0,\varepsilon),\\
    N(z)D^*(z)^{\s_3} E(z)^{-\s_3}, & z\in U(0,\varepsilon)\cap \Omega_\infty,\\
  N(z)D(z)^{\s_3} \et(z)^{-\s_3}, & z\in U(0,\varepsilon)\cap \Omega_\pm^3,\\
  \end{array}
  \right.
 \label{T_asy_Il}
\end{align}
and by \eqref{D_asy} and \eqref{Ds_asy} we obtain
\begin{align}
  T(z)\sim \left\{
  \begin{array}{ll}
   N(z)D^*(z)^{\s_3}E(z)^{-\s_3},  & z\in \Omega_\infty\setminus U(0,\varepsilon), \\
  N(z)D(z)^{\s_3} \et(z)^{-\s_3}, & z\in \Omega_\pm^3 \setminus U(0,\varepsilon),\\
    N(z)D^*(z)^{\s_3} E(z)^{-\s_3}, & z\in U(0,\varepsilon)\cap \Omega_\infty,\\
  N(z)D(z)^{\s_3} \et(z)^{-\s_3}, & z\in U(0,\varepsilon)\cap \Omega_\pm^3,\\
  \end{array}
  \right.
 \label{T_asym}
\end{align}
for $\Re z < x_0$.
From the definitions of $\et(z)$ and $D^*(z)$ in (\ref{Et}) and (\ref{Ds_D}), respectively, it is easily shown that
\begin{align}
	D^*(z)\et(z)=D(z)E(z).
 \label{Ds_et}
\end{align}
Hence, it follows that for $\Re z<x_0$,
\begin{align}
  T(z)\sim N(z) D^*(z)^{\s_3} E(z)^{-\s_3},
  \qquad z\in \Omega_\infty \cup \Omega_\pm^3
  .
  	\label{T:asy:lx0}
\end{align}
Similarly, for $\Re z>x_0$, we  have by using \r{V1}, \r{V_asyr} and \r{T_V}
\begin{align}
  T(z)\sim N(z) D^*(1-z)^{\s_3} E(z)^{-\s_3},
 \qquad z \in \Omega_\infty\cup\Omega_\pm^1,
 	\label{T:asy:gx0}
\end{align}
where we have also made use of the identity $D^*(1-z)\et(z)=D(1-z)E(z)$ obtained from \r{Et}
and \r{Ds_D}.
Define
\begin{align}
  \wt D(z):=
  \left\{
  \begin{array}{ll}
     D^*(z), & \Re z<x_0, \\
     D^*(1-z), & \Re z>x_0. \\
  \end{array}
\right.
	\label{Dt}
\end{align}
On account of \eqref{gt_def}, \eqref{T} and \eqref{E}, one can work backwards to get
\begin{align}
	R(z) &\sim
	(c_1 e^{nl} )^{\s_3/2}  N(z)
   \wt D(z)^{\s_3}e^{ n\wt g(z)\s_3 }E(z)^{ -\s_3}(c_1 e^{nl})^{- \s_3/2}
  \hh
	(c_1 e^{nl} )^{\s_3/2}  N(z)
	e^{-n\phi(z) \sigma_3}
   \bigg[
   \frac{  \wt D(z)  }{ c_1^{1/2} \prod_{j=0}^{N-1}(z-x_{N,j} )   }
   \bigg]^{\s_3}
 \label{R_asy}
\end{align}
for $z\in\Omega_\infty\cup \Omega_\pm^1\cup \Omega_\pm^3$, where $l$ is given in Definition 4.1.

\section{Construction of the Parametrices}
In this section,  we will construct an approximation $\wt R(z)$ to the solution of the RHP for $R(z)$  stated in Section 3 for large $n$.
Given $x_1\in(0,a)$, define the rectangle $K = \{z \in\mathbb{C} :\ x_1 < \Re z <1-x_1,\ |\Im z| < \delta \}$, and the line $\Gamma_0=\{z\in\mathbb{C}:\Re z=x_0, |\Im z| < \delta \}$, where $x_0$ is given in Section 5.
We divide the
complex plane into three regions I, II, III  by using $\Gamma_0$ and $K$; see Figure \ref{asy_fig}.
\begin{figure}[tbp]
\begin{center}
\includegraphics[scale=1]{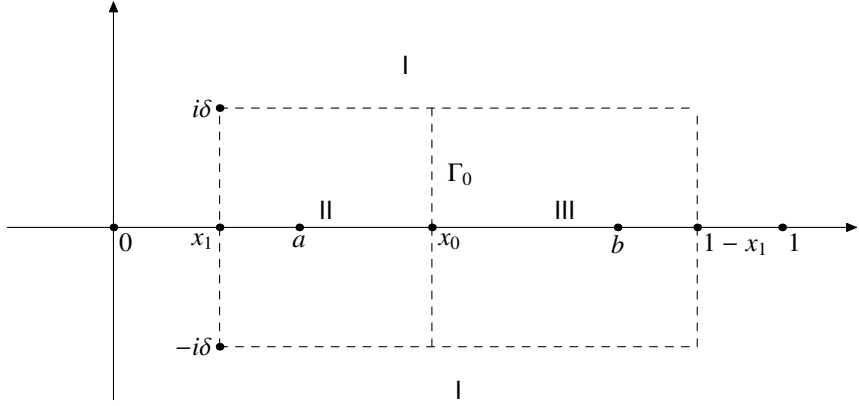}
\caption{The regions I, II, III.}\label{asy_fig}
\end{center}
\end{figure}

From  \eqref{R_asy}, it is natural to set
\begin{align}
	\wt R(z) &=
	(c_1 e^{nl} )^{\s_3/2}  N(z)
	e^{-n\phi(z) \sigma_3}
   \bigg[
   \frac{  \wt D(z)  }{ {c_1}^{1/2} \prod_{j=0}^{N-1}(z-x_{N,j} )  }
   \bigg]^{\s_3}
   \label{Rt:I}
\end{align}
for $z\in$ I.


For $z$ inside the rectangle $K$, we will construct the parametrices by
using Airy functions as in \cite{DaiWong, WangWong}.
For $z\in\mbox{II}$, it follows from  (\ref{Ds_D}), \eqref{Dt} and (\ref{R_asy})
\begin{align}
	R(z) \sim (c_1 e^{nl} )^{\s_3/2}  N(z)
	e^{  -n\phi(z)  \sigma_3}e^{\pm N\pi i z \s_3}
   \bigg[
   \frac{  2\cos(N\pi z) D(z)  }{ {c_1}^{1/2} \prod_{j=0}^{N-1}(z-x_{N,j} )   }
   \bigg]^{\s_3},
   \qquad z\in \mathbb{C}_\pm,
   \label{R_asy_pt1}
\end{align}
and from (\ref{pt_phi}) and (\ref{D_asy}) we obtain
\begin{align}
	R(z) \sim
	(-1)^n
	(c_1 e^{nl} )^{\s_3/2}  N(z)
	e^{  -n\pt(z)  \sigma_3}
   \bigg[
   \frac{  2\cos(N\pi z)  }{ {c_1}^{1/2} \prod_{j=0}^{N-1}(z-x_{N,j} )   }
   \bigg]^{\s_3}.
  \label{R_asy_pt}
\end{align}
Using  \eqref{nu} and \eqref{pt_def}, one can derive an expansion for $\pt(z)$ for $z\in U(a,\varepsilon):=\{z\in \mathbb{C}:|z-a|<\varepsilon \}$.
Indeed, we have
\begin{align}
 \pt(z)=-\frac{8}{3 c^2}(b-a)^{1/2} (a-z)^{3/2} + O(\varepsilon^2)
 \label{pt:a:exp}
\end{align}
for $z\in U(a,\varepsilon)$.
By using (4.29) and the proposition of $\pt$ (see Figure \ref{pt_fig}),  it can be verified that the function defined by
\begin{align}
	\ft(z)=\( -\frac32  \pt(z) \)^{2/3}
	\label{ft_def}
\end{align}
satisfies
$$
	\ft_-(x)^{-1}\ft_+(x)=1,
	\qquad \qquad x\in(a,b).
$$
On account of this and \r{pt:a:exp}, we conclude that this function
is analytic in $\mathbb{C}\setminus (-\infty, 0]\cup[b,+\infty)$.
In particular, one can see that for $z\in\mbox{II} \cap \mathbb{C}_+$
\begin{align}
	-\pi<\arg \ft(z) <0,
	\label{ft_arg_upp}
\end{align}
and  we obtain from \eqref{pt_map_ab}
\begin{align}
	\ft_+(x)^{1/4}\ft_-(x)^{-1/4}=e^{-\pi i/2},
	\qquad x\in(a,b)
	.
 \label{ft_jum}
\end{align}
Note that $N(z)=M_\infty^{-\s_3}\wt N(z)M(z)^{\s_3} $ by \r{Nt} and \r{N}, and $\wt N(z)$ has the factorization
\begin{align}
	\wt N(z)&= \frac12
	\begin{pmatrix}
	1  &  i  \\
	i  &  1
	\end{pmatrix}
	a(z)^{-\s_3}
	\begin{pmatrix}
	1  &  -i  \\
	-i  &  1
	\end{pmatrix}
 \hh
 	\frac12
	\begin{pmatrix}
	1  &  -i  \\
	-i  &  1
	\end{pmatrix}
	a(z)^{\s_3}
	\begin{pmatrix}
	1  &  i  \\
	i  &  1
	\end{pmatrix}.
	\label{Nt_fac}
\end{align}
Since $a(z)={ (z-b)^{1/4} }/{(z-a)^{1/4}}$,
$N(z)$ has fourth-root singularities at $z=a$ and $z=b$.
In order to remove the singularity at $z=a$,
we introduce the function
\begin{align}
	\wt E_n(z)=  N(z) h(z)^{\s_3/2}
	\begin{pmatrix}
	1 &  i \\
	i &  1
\end{pmatrix}
	[n^{2/3} \ft(z)]^{-\s_3/4}
.
	\label{Etn}
\end{align}
Note that by using \r{ab}, we have from \r{M}
$$
	\lim_{z\ra a} M(z) = a^{-\a/2} b^{-\b/2}.
$$
Also, from \r{h} we have
$$
	\lim_{z\ra a} h(z) = a^{\a} b^{\b}.
$$
Thus,
$$
	M(z)^{\s_3} h(z)^{\s_3/2} \ra I
	\qquad \qquad \mbox{as } z\ra a.
$$
Coupling \r{Etn} and \r{Nt_def}, we obtain
\begin{align}
	\et_n(z) \sim
	 M_\infty^{-\s_3}
	\begin{pmatrix}
	1  &  i  \\
	i  &  1
	\end{pmatrix}
	\( n^{1/6} a(z) \ft(z)^{1/4} \)^{-\s_3}
	\qquad \qquad
	\mbox{as } z\ra a.
 \label{Etn:ten}
\end{align}
Since $\ft(z)$ has a simple zero at $z=a$ and $a(z)$ has a fourth-root singularity at $a$, the last formula shows that $\et_n(z)$ has a removable singularity at $z=a$.

Furthermore, one can show that $\et_n(z)$ has no jump across the interval $(a,b)$. Indeed, by \eqref{N_jum}
\begin{align}
	N_+(x) h(x)^{\s_3/2}
 &=N_-(x)
	\begin{pmatrix}
	0  &  -h(x)  \\
	h(x)^{-1}  &  0
	\end{pmatrix}
	h(x)^{\s_3/2}
 \hh
 	N_-(x)
 	h(x)^{\s_3/2}
 	\begin{pmatrix}
	0  &  -1  \\
	1  &  0
	\end{pmatrix},
	\qquad x\in(a,b)
	.
 \label{N:jum:Etn}
\end{align}
From \eqref{ft_jum}, \eqref{Etn} and \eqref{N:jum:Etn}, it follows that for $x\in(a,b)$
\begin{align}
&\qquad	[ ( \wt E_n )_- (x) ]^{-1} ( \wt E_n )_+ (x)
\hh	\frac12
	[n^{2/3}\ft_-(x)]^{\s_3/4}
	\begin{pmatrix}
	1 &  -i \\
	-i &  1
\end{pmatrix}
	\begin{pmatrix}
	0 &  -1 \\
	1 &  0
\end{pmatrix}
	\begin{pmatrix}
	1 &  i \\
	i &  1
\end{pmatrix}
	[n^{2/3}\ft_+(x)]^{-\s_3/4}
 \hh
	[n^{2/3}\ft_-(x)]^{\s_3/4}
	\begin{pmatrix}
	-i &  0 \\
	0 &  i
\end{pmatrix}
	[n^{2/3}\ft_+(x)]^{-\s_3/4}
=
 	I.
 	\label{E0t}
\end{align}
Therefore, $\et_n(z)$ is analytic across $(a,b)$. Also, note that from \r{N}, \r{h} and \r{ft_def},
it is evident that  $N(z)$, $h(z)$ and $\ft(z)$ are analytic for $z\in(0,a)$.
Thus, the matrix-valued function $\wt E_n(z)$ is analytic
in $\mathbb{C}\setminus (-\infty,0]\cup [b,+\infty)$.

By virtue of  \eqref{Etn}, we can rewrite (\ref{R_asy_pt}) as
\begin{align}
	R(z) &\sim
	(-1)^n (c_1 e^{nl} )^{\s_3/2} \et_n(z)\et_n(z)^{-1}
	N(z) e^{ -n\pt(z)\s_3 }
	   \left[
   \frac{  2\cos(N\pi z)  }{ {c_1} ^{1/2} \prod_{j=0}^{N-1}(z-x_{N,j} )   }
   \right]^{\s_3}
	\notag
	\\
	&\sim
 	\frac{(-1)^n}{2}
 	(c_1 e^{nl} )^{\s_3/2}
	\wt E_n(z)
	[n^{2/3} \ft(z)]^{\s_3/4}
	\begin{pmatrix}
	e^{  -n\pt(z)  } &  -ie^{  n\pt(z) } \\
	-ie^{  -n\pt(z)  } & e^{  n\pt(z) }
\end{pmatrix}
   \left[
   \frac{  2\cos(N\pi z)  }{ (c_1 h(z))^{\frac12} \prod_{j=0}^{N-1}(z-x_{N,j} )   }
   \right]^{\s_3}.
  \label{R_asy_ft1}
\end{align}
To find an approximation to $R(z)$, we first look for a matrix which is asymptotic to
\begin{align}
	[n^{2/3} \ft(z)]^{\s_3/4}\begin{pmatrix}
	e^{  -n\pt(z)  } &  -ie^{  n\pt(z) } \\
	-ie^{  -n\pt(z)  } & e^{  n\pt(z) }
\end{pmatrix}.
\end{align}
Set
\begin{align}
	\wt \xi :=n^{2/3}\ft(z)=\( -\frac32n  \pt(z) \)^{2/3}
	.
 \label{xt}
\end{align}
From the asymptotic expansions of the Airy function \cite[p.198]{Handbook}, we have
\begin{equation}
\begin{aligned}
	&\Ai(   \wt \xi  ) \sim \frac{\xt^{-1/4}}{ 2\sqrt{\pi} }
	e^{ -\frac23  \wt \xi^{3/2}  }
	=
	\frac{1}{ 2\sqrt{\pi} }
	[n^{2/3} \ft(z)]^{-1/4 } e^{  n \pt(z) },
\\
 	&\Ai'(   \wt \xi ) \sim -\frac{\xt^{1/4}}{ 2\sqrt{\pi} }
	 e^{ -\frac23  \wt \xi^{3/2}  }
	=
	-\frac{1}{ 2\sqrt{\pi} }
	[n^{2/3}\ft(z)]^{1/4 } e^{  n \pt(z) },
\\
 	&\Ai(   \om \wt \xi )
 	\sim
 	\frac{ e^{-i\pi/6 } }{ 2\sqrt{\pi} }
	\wt \xi^{-1/4 } e^{  \frac23  \wt\xi^{3/2} }
 	=
	\frac{ e^{-i\pi/6 } }{ 2\sqrt{\pi} }[n^{2/3}\ft(z)]^{-1/4}
	 e^{  -n \pt(z) },
\\
 	&\Ai'(   \om \wt \xi )
 	\sim
 	-\frac{ e^{i\pi/6 } }{ 2\sqrt{\pi} }
	\wt \xi^{1/4 } e^{  \frac23  \wt\xi^{3/2} }
 	=
	-\frac{ e^{i\pi/6 } }{ 2\sqrt{\pi} }
	[n^{2/3}\ft(z)]^{1/4 } e^{  -n \pt(z) },
\end{aligned}
 \label{Ai:1}
\end{equation}
where $\om=e^{2\pi i /3}$. It is immediate that for $z\in\mbox{II} \cap \mathbb{C}_+$,
\begin{align}
	[n^{2/3} \ft(z)]^{\s_3/4}\begin{pmatrix}
	e^{  -n\pt(z)  } &  -ie^{  n\pt(z) } \\
	-ie^{  -n\pt(z)  } & e^{  n\pt(z) }
\end{pmatrix}
\sim
	2\sqrt \pi
	\begin{pmatrix}
	-i\om^2 \Ai'( \om\wt \xi) &  i\Ai'( \wt \xi) \\
	-\om\Ai( \om \wt \xi) &  \Ai( \wt \xi)
\end{pmatrix}.
 \label{R_mat_2_plu}
\end{align}
Similarly, for $z\in\mbox{II}\cap \mathbb{C}_-$, we also have $0<\arg\ft(z)<\pi$
and
\begin{align}
	[n^{2/3} \ft(z)]^{\s_3/4}\begin{pmatrix}
	e^{  -n\pt(z)  } &  -ie^{  n\pt(z) } \\
	-ie^{  -n\pt(z)  } & e^{  n\pt(z) }
\end{pmatrix}
\sim
	2\sqrt \pi
	\begin{pmatrix}
	i\om \Ai'( \om^2\wt \xi) &  i\Ai'( \wt \xi) \\
	\om^2\Ai( \om^2\wt \xi) &  \Ai( \wt \xi)
\end{pmatrix}.
 \label{R_mat_2_min}
\end{align}

For $z\in\mbox{III}$, there is a result corresponding to (\ref{R_asy_pt}) with $\pt(z)$ replaced
by $\ps(z)$. Indeed, it follows from (\ref{Dt}), (\ref{R_asy}) and (\ref{ps_phi}) that
\begin{align}
	R(z) \sim
	(-1)^N
	(c_1 e^{nl} )^{\s_3/2}  N(z)
	e^{  -n\ps(z)  \sigma_3}
   \bigg[
   \frac{  2\cos(N\pi z)  }{ {c_1}^{1/2} \prod_{j=0}^{N-1}(z-x_{N,j} )   }
   \bigg]^{\s_3}.
  \label{R_asy_ps}
\end{align}
Let
\begin{align}
	\fs(z) =\( -\frac32 \ps(z)  \)^{2/3},
	\label{fs_def}
\end{align}
which is analytic in $\mathbb{C}\setminus (-\infty,a]\cup [1,+\infty)$.
By Proposition 4.3, we see that for $x\in(a,b)$
\begin{align}
	\fs_+(x)^{1/4}\fs_-(x)^{-1/4}=e^{\pi i/2}
	.
 \label{fs_jum}
\end{align}
Since $N(z)$  has a fourth-root singularity at $z=b$,  we define
\begin{align}
	 E^*_n(z)=  N(z) h(z)^{\s_3/2}
	\begin{pmatrix}
	1 & - i \\
	-i &  1
\end{pmatrix}
	[ n^{2/3}\fs(z)]^{-\s_3/4}
.
	\label{Esn}
\end{align}
As in the case of $\et_n(z)$,
it is readily seen that  $b$ is a removable singularity of $E_n^*(z)$. By (6.11) and  (\ref{fs_jum}), one also has
\begin{align}
	[ (  E_n^* )_- (x) ]^{-1} (  E_n^* )_+ (x)
&=
	[n^{2/3} \fs_-(z)]^{\s_3/4}
	\begin{pmatrix}
	i &  0 \\
	0 &  -i
\end{pmatrix}
	[ n^{2/3} \fs_+(z)]^{-\s_3/4}
 \hh
 	I
 	\label{Eb}
\end{align}
for $x\in(a,b)$.
Therefore, $E_n^*(z)$ is analytic in $\mathbb{C}\setminus(-\infty,a]\cup[1,+\infty)$.
On account of this and (\ref{R_asy_ps}), we have for $z\in\mbox{III}$
\begin{align}
	R(z)
	&\sim
	(-1)^N (c_1 e^{nl} )^{\s_3/2} E^*_n(z) E^*_n(z)^{-1}
	N(z) e^{ -n\ps(z)\s_3 }
	   \left[
   \frac{  2\cos(N\pi z)  }{ {c_1}^{1/2} \prod_{j=0}^{N-1}(z-x_{N,j} )   }
   \right]^{\s_3}
	\notag
	\\
	&\sim
 	\frac{(-1)^N}{2}
 	(c_1 e^{nl} )^{\s_3/2}
	 E^*_n(z)
	[ n^{2/3} \fs(z)]^{\s_3/4}
	\begin{pmatrix}
	e^{  -n\ps(z)  } &  ie^{  n\ps(z) } \\
	ie^{  -n\ps(z)  } & e^{  n\ps(z) }
\end{pmatrix}
   \bigg[
   \frac{  2\cos(N\pi z)  }{ (c_1 h(z))^{1/2} \prod_{j=0}^{N-1}(z-x_{N,j} )   }
   \bigg]^{\s_3}.
  \label{R_asy_fs1}
\end{align}
Hence, as before, it can be shown that the matrix
\begin{align}
	[n^{2/3} \fs(z)]^{\s_3/4}
	\begin{pmatrix}
	e^{  -n\ps(z)  } &  ie^{  n\ps(z) } \\
	ie^{  -n\ps(z)  } & e^{  n\ps(z) }
	\end{pmatrix}
\end{align}
is the leading term of the asymptotic expansion of the matrices
\begin{align}
	2\sqrt{\pi} \begin{pmatrix}
	i\om \Ai'( \om^2 \xi^* ) & -i\Ai'(  \xi^*)  \\
	-\om^2 \Ai(  \om^2 \xi^*) &  \Ai(  \xi^* )
\end{pmatrix}
	\label{ps:17}
\end{align}
for $z\in\mbox{III} \cap \mathbb{C}_+$, and
\begin{align}
	2\sqrt{\pi} \begin{pmatrix}
	-i\om^2 \Ai'( \om \xi^* ) & -i\Ai'(  \xi^*)  \\
	\om \Ai(  \om \xi^*) &  \Ai(  \xi^* )
\end{pmatrix}
 	\label{ps:18}
\end{align}
for $z\in\mbox{III} \cap \mathbb{C}_-$,
where $\xs=n^{2/3}\fs(z)=(-\frac32 n \ps(z))^{2/3} $.
Define the matrix function
\begin{equation}
  P(z):=\left\{\begin{array}{ll}
  \begin{pmatrix}
	-i\om^2 \Ai'( \om\wt \xi) &  i\Ai'( \wt \xi) \\
	-\om\Ai( \om \wt \xi) &  \Ai( \wt \xi)
\end{pmatrix},  & z\in\mbox{II}\cap \mathbb{C}_+,
\\\\
  \begin{pmatrix}
	i\om \Ai'( \om^2\wt \xi) &  i\Ai'( \wt \xi) \\
	\om^2\Ai( \om^2\wt \xi) &  \Ai( \wt \xi)
\end{pmatrix},  & z\in\mbox{II}\cap \mathbb{C}_-,
\\\\
  \begin{pmatrix}
	i\om \Ai'( \om^2 \xi^* ) & -i\Ai'(  \xi^*)  \\
	-\om^2 \Ai(  \om^2 \xi^*) &  \Ai(  \xi^* )
\end{pmatrix},  & z\in\mbox{III}\cap \mathbb{C}_+,
\\\\
  \begin{pmatrix}
	-i\om^2 \Ai'( \om \xi^* ) & -i\Ai'(  \xi^*)  \\
	\om \Ai(  \om \xi^*) &  \Ai(  \xi^* )
\end{pmatrix},  & z\in\mbox{III}\cap \mathbb{C}_-.
  \end{array}\right.
 \label{P}
\end{equation}
An appeal to the formulas
\begin{equation}
\begin{array}{l}
	\Ai(z)+\om \Ai(\om z)+\om^2 \Ai(\om^2 z)=0,
	\\
	\Ai'(z)+\om^2 \Ai'(\om z)+\om \Ai'(\om^2 z)=0
\end{array}
	\label{airy_con}
\end{equation}
shows  $P(z)$ satisfies
\begin{align}
	P_+(x)=P_-(x)
 \begin{pmatrix}
	1 &  0 \\
	1 &  1
\end{pmatrix},
\qquad x\in\mathbb{R}.
 \label{P_jum}
\end{align}

Let $\et_n(z)$ and $E_n^*(z)$ be defined as in \eqref{Etn} and \eqref{Esn}, respectively. In view of the  heuristic argument leading to \r{Rt:I}, \r{R_asy_ft1} and \r{R_asy_fs1}, it is reasonable to suggest that $R(z)$ is asymptotically approximated
by
\begin{equation}
 \wt R(z):=\left\{
 	\begin{array}{ll}
 	(c_1 e^{nl} )^{\s_3/2}  N(z)
	e^{-n\phi(z) \sigma_3}M_1(z), & z\in\mbox{I}, \\\\
	(-1)^n \sqrt{\pi}
	(c_1 e^{nl} )^{\s_3/2}  \et_n(z)P(z)M_2(z),  &  z\in\mbox{II},\\\\
	(-1)^N \sqrt{\pi}
	(c_1 e^{nl} )^{\s_3/2}  E^{*}_n(z)P(z)M_2(z),  &  z\in\mbox{III},
 	\end{array}
 \right.
 \label{Rt}
\end{equation}
where
	\begin{align}
	M_1(z):=
		\bigg[
   	\frac{  \wt D(z)  }{ {c_1}^{1/2} \prod_{j=0}^{N-1}(z-x_{N,j} )  }
   	\bigg]^{\s_3}
   	 \label{M1}
	\end{align}
for $z\in\mbox{I}$, and
	\begin{align}
		M_2(z):=
			\bigg[
   		\frac{  2\cos(N\pi z)  }{ (c_1h(z))^{1/2} \prod_{j=0}^{N-1}(z-x_{N,j} )  }
   			\bigg]^{\s_3}
   				\label{M2}
	\end{align}
for $z\in\mbox{II and III}$.
From \eqref{Rt}, it can be shown that
$\wt R(z)$ has the same large $z$ behavior as $R(z)$  given in
($R_c$). To see this, we only consider the case when $z\in\mbox{I}$ and $\Re z >x_0$.
The discussion for the case $z\in \mbox{I}$ and $\Re z <x_0$ is  very similar.
Note that when $z\in\mbox{I}$ and $\Re z>x_0$,
	\begin{align}
		\wt	R(z) = (c_1 e^{nl} )^{\s_3/2}  N(z)
		e^{-n\phi(z) \sigma_3}
   		\bigg[
   		\frac{  \wt D(z)  }{ {c_1}^{1/2} \prod_{j=0}^{N-1}(z-x_{N,j} )   }
   		\bigg]^{\s_3}.
	\end{align}
From \eqref{Dt} and \r{Ds_asy}, it is easily seen that $\wt D(z)=D^*(1-z) \sim 1$ for large $z$. Furthermore, from  \eqref{phi_def} and  $\prod_{j=0}^{N-1}(z-x_{N,j} ) \sim z^N$,  we have
\begin{align}
	\wt	R(z) &\sim (c_1 e^{nl} )^{\s_3/2}  N(z)
   	(c_1 e^{nl} )^{-\s_3/2}
	e^{n g(z) \sigma_3}z^{-N\s_3}.
 \label{Rt:inf:1}
\end{align}
This together with (\ref{g_def}) and the asymptotic behavior of $N(z)$ in ($N_c$) gives
\begin{align}
	\wt	R(z) &\sim (c_1 e^{nl} )^{\s_3/2}  N(z)
   	(c_1 e^{nl} )^{-\s_3/2}
	e^{n g(z) \sigma_3}z^{-N\s_3}
   \hh
   	(c_1 e^{nl} )^{\s_3/2} [I+O(1/z)]
   	(c_1 e^{nl} )^{-\s_3/2}
   	z^{(n-N)\s_3}
   \hh
   	[I+O(1/z)]
   	\begin{pmatrix}
	z^{n-N} &  0 \\
	0 &  z^{-n+N}
\end{pmatrix}
\qquad \qquad
\mbox{as } z\ra\infty.
 \label{Rt:inf}
\end{align}

Let $\Gamma:=\bigcup_{i=0}^{8}\Gamma_i\cup (0,1)$; see Figure \ref{Js_fig}. Since $\wt R(z)$ has the same large $z$-behavior as $R(z)$, it can be readily verified that the matrix-valued function
\begin{align}
	S(z):=(c_1 e^{nl} )^{-\s_3/2} R(z)\wt R(z)^{-1}
	(c_1 e^{nl} )^{\s_3/2}
 \label{S_def}
\end{align}
is a solution of the RHP:
\begin{enumerate}
\item[($S_{a}$)]
$S(z)$ is analytic for $z\in \mathbb{C}\setminus \Gamma$;
\item[($S_b$)] for $z\in \Gamma$,
\begin{equation*}
    S_+(z)
=
    S_-(z)
    J_S(z),
\end{equation*}
where
\begin{align}
J_S(z):=(c_1 e^{nl} )^{-\s_3/2} \wt R_-(z)
	J_R(z)
	\wt R_+(z)^{-1}
	(c_1 e^{nl} )^{\s_3/2}
	;
 \label{Js}
\end{align}
\item[($S_c$)] for $z\in\mathbb{C}\setminus \Gamma$,
\begin{align}
	S(z)\rightarrow I
	\qquad \mbox{as } z\rightarrow \infty,
\end{align}
\end{enumerate}
where the contour associated with the matrix $J_S(z)$ is $\Gamma=\Sigma\cup\Gamma_0\cup\partial K$; see Figure 1 and Figure 6.
		\begin{figure}[!tpb]
		\centering
		\includegraphics[scale=1.1]{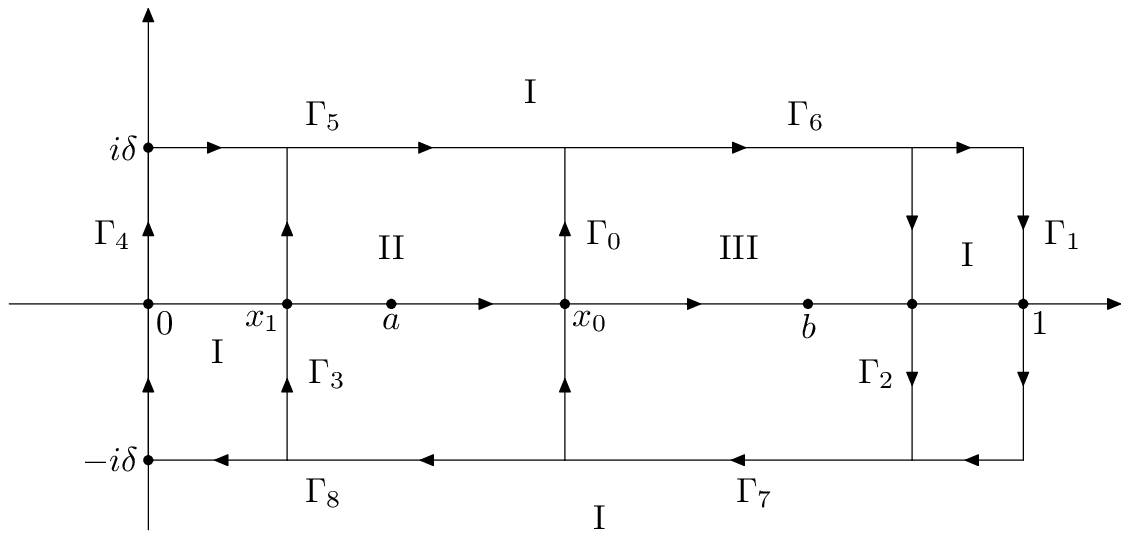}
		\caption{Contour $\Gamma$.  }
		\label{Js_fig}
		\end{figure}

To solve this problem, we first derive the asymptotic behavior of the jump matrix $J_S(z)$  as $n\rightarrow \infty$.
For convenience, we  consider only the cases when $z$ lies on the contour $\Gamma_0\cap \cp$ and $ \Gamma_3\cap \cp$, and $z$ in the interval $(0,1)$,
since the other cases can be discussed
in a similar manner.

For $z\in\Gamma_0\cap \cp$, we have from \eqref{Rt}, \r{P} and \eqref{airy_con}
\begin{align}
	 \wt R_+(z)&=(-1)^n\sqrt \pi (c_1 e^{nl} )^{\s_3/2}
	 \wt E_n(z) \begin{pmatrix}
	-i\om^2 \Ai'( \om\wt \xi) &  i\Ai'( \wt \xi) \\
	-\om\Ai( \om \wt \xi) &  \Ai( \wt \xi)
\end{pmatrix}M_2(z)
	\hh
		(-1)^n\sqrt \pi (c_1 e^{nl} )^{\s_3/2} \wt E_n(z) \begin{pmatrix}
	-i\om^2 \Ai'( \om\wt \xi) &
	-i\om \Ai'( \om^2\wt \xi) \\
	-\om\Ai( \om \wt \xi) &
	-\om^2 \Ai( \om^2 \wt \xi)
\end{pmatrix}
	\begin{pmatrix}
	1 &  1 \\
	0 &  1
\end{pmatrix}
M_2(z).
\label{Js_1_fac1}
\end{align}
Similarly,
\begin{align}
	 \wt R_-(z)
	 &=
	(-1)^N \sqrt \pi (c_1 e^{nl} )^{\s_3/2} E_n^*(z)  \begin{pmatrix}
	i\om \Ai'( \om^2 \xi^* ) & i\om^2 \Ai'( \om \xi^*)
	 \\
	-\om^2 \Ai(  \om^2 \xi^*) &  -\om\Ai( \om  \xi^*)
\end{pmatrix}
	\begin{pmatrix}
	1 &  1 \\
	0 &  1
\end{pmatrix}
M_2(z).
\label{Js_1_fac2}
\end{align}
Since $R(z)$ is analytic in $\mathbb{C}\setminus\Sigma$, $J_R(z)=1$ for $z\in\Gamma_0\cap\cp$;
see the RHP for $R(z)$ in Section 3. Thus,
from (\ref{Js}) it follows that
\begin{align}
	J_S(z)&=
		(-1)^{N-n}   E_n^*(z)
		\begin{pmatrix}
	i\om \Ai'( \om^2 \xi^* ) & i\om^2 \Ai'( \om \xi^*)
	 \\
	-\om^2 \Ai(  \om^2 \xi^*) &  -\om\Ai( \om  \xi^*)
	\end{pmatrix}
	 \begin{pmatrix}
	-i\om^2 \Ai'( \om\wt \xi) &
	-i\om \Ai'( \om^2\wt \xi) \\
	-\om\Ai( \om \wt \xi) &
	-\om^2 \Ai( \om^2 \wt \xi)
\end{pmatrix} ^{-1} \ent^{-1}.
 \label{JS:asyG}
\end{align}
On account of the well-known formula  \cite[p.194, (9.2.9)]{Handbook}
\bb
	\om \Ai(\om^2 \xt)\Ai'(\om\xt)-\om^2\Ai(\om\xt)\Ai'(\om^2 \xt)
	=\frac{1}{2\pi i},
 \label{Ai_rel}	
\ee
we obtain
\begin{align}
	J_S(z)&=
		(-1)^{N-n}  2\pi   E_n^*(z)
		\begin{pmatrix}
	i\om \Ai'( \om^2 \xi^* ) & i\om^2 \Ai'( \om \xi^*)
	 \\
	-\om^2 \Ai(  \om^2 \xi^*) &  -\om\Ai( \om  \xi^*)
	\end{pmatrix}
	 \begin{pmatrix}
	-\om^2 \Ai( \om^2\xt ) &
	i\om \Ai'( \om^2\wt \xi) \\
	\om\Ai( \om \wt \xi) &
	-i\om^2 \Ai'( \om \wt \xi)
	\end{pmatrix}
	\ent^{-1}.
 \label{JS:asyG:inv}
\end{align}

Now, we set out to derive the asymptotic behavior of $J_S(z)$ for $z\in\Gamma_0\cap \cp$.  Since by \eqref{pt_map_ab} and \eqref{ft_def} $\arg \ft(x_0)=-\pi$, $\arg \wt \xi=\arg \ft(z) \in [-\pi,-\pi/3)$ for $z\in\Gamma_0\cap \cp$. We recall the asymptotic behavior of the Airy function and its
derivative in \cite[p.198]{Handbook}, and have as in (6.16)
\begin{equation}
\begin{array}{ll}
 -\om^2 \Ai(\om^2 \xt)
 =
 \frac{ \xt^{-1/4} }{ 2\sqrt\pi }e^{ n\pt(z) }
 (1+O(1/n)),
&
 \qquad i\om \Ai'(\om^2 \xt)
 	=\frac{  i\xt^{1/4}	}{  2\sqrt\pi	} e^{ n\pt(z) }
  (1+O(1/n)),
 \\\\
 	\om \Ai(\om \xt)
 	= \frac{  i\xt^{-1/4}	}{  2\sqrt\pi	}e^{ -n\pt(z) }
 (1+O(1/n)),
 &\qquad
 	-i\om^2 \Ai'(\om \xt)
 	= \frac{  \xt^{1/4}	}{  2\sqrt\pi	} e^{ -n\pt(z) }
 	(1+O(1/n)).
 \end{array}
 \label{Js:1:xt:asy}
\end{equation}
Corresponding results can be given for $\Ai(\om \xs)$, $\Ai'(\om \xs)$, $\Ai(\om^2 \xs)$ and
$\Ai'(\om^2 \xs)$.
Substituting  \r{Esn}, \r{Etn} and the above asymptotic formulas into \eqref{JS:asyG:inv},  and taking into account (4.36), we readily see that
	\begin{align}
		J_S(z)&=
		(-1)^{N-n} \sqrt \pi
		N(z) h(z)^{\s_3/2}
	\begin{pmatrix}
	1 & - i \\
	-i &  1
	\end{pmatrix}
 	\begin{pmatrix}
	1+O(1/n) & i+O(1/n)  \\
	i+O(1/n) & 1+O(1/n)
	\end{pmatrix} e^{ -n\ps(z)\s_3 }
	\hhh
	\times
		\frac{1}{4\sqrt\pi}e^{ n\pt(z)\s_3 }
 	\begin{pmatrix}
		1+O(1/n) & i+O(1/n) \\
		i+O(1/n) & 1+O(1/n)
	\end{pmatrix}
	\begin{pmatrix}
	1 &  -i \\
	-i &  1
\end{pmatrix}
	h(z)^{-\s_3/2} N(z)^{-1}
	\hh
			(-1)^{N-n}
	    N(z) h(z)^{\s_3/2}
	 e^{ n(\pt-\ps)\s_3 }
	h(z)^{-\s_3/2} N(z)^{-1}[I+O(1/n)]
	 \hh
	 	I+O(1/n).
		\label{Js:1:asy}
	\end{align}

For $z\in\Gamma_3\cap \cp$, we again have $J_R(z)=I$ and  it follows from \r{Js}, (\ref{P}) and \eqref{Rt}  that
\begin{align}
	J_S(z)
	&=
	(-1)^n \sqrt \pi
	\wt E_n(z) \begin{pmatrix}
	-i\om^2 \Ai'( \om\wt \xi) &  i\Ai'( \wt \xi) \\
	-\om\Ai( \om \wt \xi) &  \Ai(\wt \xi)
\end{pmatrix}M_2(z)
		  M_1(z)^{-1} h(z)^{\s_3/2}
		  \hhh \times
		  e^{ n\phi(z)\s_3 }
		  h(z)^{-\s_3/2}N(z)^{-1}.
 \label{Js:2:inv}	
\end{align}
Moreover, we note from \eqref{M1}-\eqref{M2}, \eqref{Dt} and \eqref{Ds_D} that
	\begin{align}
	M_2(z)M_1(z)^{-1} h(z)^{\s_3/2}
		=e^{-N\pi i z\s_3} D(z)^{-\s_3}
		\label{Js:2:M12},
	\end{align}
and by \eqref{Ai:1} and \eqref{Etn} we obtain
\begin{align}
		\et_n(z)\begin{pmatrix}
	-i\om^2 \Ai'( \om\wt \xi) &  i\Ai'( \wt \xi) \\
	-\om\Ai( \om \wt \xi) &  \Ai( \wt \xi)
		\end{pmatrix}
		&=
			\frac{1}{ \sqrt{\pi} }N(z) h(z)^{\s_3/2}
		[I+O(1/n)]
	e^{-n\pt(z)\s_3}.
	 \label{Js:2:kk}
	\end{align}
A combination of \eqref{Js:2:inv}-\eqref{Js:2:kk}, \eqref{pt_phi} and
\eqref{D_asy} yields
	\begin{align}
		J_S(z)&=
		  	(-1)^n N(z) h(z)^{\s_3/2}
		[I+O(1/n)]
	e^{-n\pt(z)\s_3}e^{-N\pi iz\s_3} D(z)^{-\s_3}
		\hhh \times
		  e^{ n\phi(z)\s_3 }
		  h(z)^{-\s_3/2}N(z)^{-1}
		 \hh
		 	I+O(1/n).
		\label{Js:2:asy}
	\end{align}

Let us now consider $z$ in the interval $(0,1)$ on the real line. For simplicity, we only consider the case $z\in(0,x_0]$.  For $z=x\in(0,x_1]$, we note by \r{Dt} that $\wt D(z)=D^*(z)$, where $x_1$ is the point of intersection of the vertical line $\Gamma_3$ with the real axis; see Figure 7. It follows from \eqref{R_real}-\eqref{r}, \r{Rt} and \r{Js} that
	\begin{align}
		J_S(x)&=N(x)
		e^{-n\phi_-(x)\s_3 }  \left[
		 	\frac{ D^*_-(x) }
		 	{c_1^{1/2} \prod^{N-1}_{j=0}(x-x_{N,j}  ) }
		\right]^{\s_3}
			\begin{pmatrix}
			1  &  0  \\
		\dfrac{4\cos^2(N\pi x)}{ c_1 w(x)
		 \prod^{N-1}_{j=0}(x-x_{N,j}  )^2 }  &  1
		\end{pmatrix}
\bigskip
	  \hhh	
	   \qquad \qquad \qquad \qquad   \times
	   	\left[
		 	\frac{ D^*_+(x) }
		 	{c_1^{1/2} \prod^{N-1}_{j=0}(x-x_{N,j}  ) }
		\right]^{-\s_3}
		e^{n\phi_+(x)\s_3 } N(x)^{-1}
            \hh
           	N(x)
	  \begin{pmatrix}
			{ D^*_- }/{D^*_+ }e^{n(\phi_+-\phi_-)}
			&  0  \\
		\dfrac{4\cos^2(N\pi x)}{ w(x)D_+^*(x) D_-^*(x) }   e^{n(\phi_++\phi_-)}
		&  {D^*_+ }/{ D^*_- }  e^{n(\phi_--\phi_+) }
		\end{pmatrix}
		N(x)^{-1}.
		\label{Js:3:pre1}
	\end{align}
From \r{Ds_D}, one can  see  that $D_+^*/D_-^*=e^{ 2N\pi i x }$ and
$
	D_+^*(x)D_-^*(x)=4\cos^2(N\pi x) D(x)^2.
$
Since $\pt(z)$ is analytic for $z\in(0,a)$,
we obtain from  \eqref{pt_phi}
	\begin{align}
		J_S(z)=
		N(x)
		\begin{pmatrix}
	1  &  0  \\
	e^{2n\pt(x) } w(x)^{-1}D(x)^{-2}  &  1
	\end{pmatrix}
	 N(x)^{-1}.
	  \label{Js:3:pre2}
	\end{align}
In fact, it is easily verified that $D(x)$ is bounded and $w(x)^{-1}=O(N^{\a+\b}) $ in this case. Since $\pt(x)<0$ for $x\in(0,x_1]$;
see \r{pt_map_0} and Figure \ref{pt_fig}, $e^{2n\pt(x) } w(x)^{-1}D(x)^{-2}= O(N^{\a+\b} e^{2n\pt} ) \sim 0$ as $n\ra \infty$.
Thus,
	$
		J_S(z)= I+ O(1/n).
	 \label{Js:3}
	$

Finally, we consider the case $z=x\in(x_1,x_0]$. By \r{Rt}, \r{M2} and \r{r}, we have
from \r{Js}
	\begin{align}
		J_S(x)&=
	 	  \et_n(x) P_-(x)
		\left[
		 	\frac{ 2\cos(N\pi x) }
		 	{ (c_1 h(x) )^{1/2} \prod^{N-1}_{j=0}(x-x_{N,j}  ) }
		\right]^{\s_3}
			\begin{pmatrix}
			1  &  0  \\
		\dfrac{4\cos^2(N\pi x)}{ c_1 w(x)
		 \prod^{N-1}_{j=0}(x-x_{N,j}  )^2 }  &  1
		\end{pmatrix}
	  \hhh
	   \times
	   	 		\left[
		 	\frac{ 2\cos(N\pi x) }
		 	{ (c_1 h(x) )^{1/2} \prod^{N-1}_{j=0}(x-x_{N,j}  ) }
		\right]^{-\s_3}
		P_+(x)^{-1} \et_n(x)^{-1}
	  \hh
	  	\et_n(x) P_-(x)
		\begin{pmatrix}
			1  &  0  \\
		h(x)w(x)^{-1}
		&  1
		\end{pmatrix}
		P_+(x)^{-1} \et_n(x)^{-1} .
		\label{Js:4:pre}
	\end{align}
Since by \r{w_asy} $w(x)\sim h(x)$ for $x\in(x_1,x_0]$, it follows from \r{P_jum} that
\begin{align}
		J_S(x)&=
			\et_n(x) P_-(x)
			\begin{pmatrix}
			1  &  0  \\
		h(x)w(x)^{-1}  &  1
		\end{pmatrix}
		\begin{pmatrix}
	1  &  0  \\
	-1  &  1
	\end{pmatrix} P_-(x)^{-1} \et_n(x)^{-1}
	 \hh
	 	 \et_n(x) P_-(x)
			\begin{pmatrix}
			1  &  0  \\
		h(x)w(x)^{-1} -1 &  1
		\end{pmatrix}
		P_-(x)^{-1} \et_n(x)^{-1}
	 \hh
	 	I+(1/n)
		 .
	 \label{Js:4:asy}
	\end{align}

In a similar manner, we can prove that for $z$ in other parts of the contour $\Gamma$,  the jump matrix $J_S(z)$ also tends to the identity matrix. Hence, $J_S (z) = I+O(1/n)$ as $n \ra \infty$ on the whole contour
$\Gamma$.
By Theorem 3.8 in \cite{QiuWong},  the solution of the RHP for $S$ has the asymptotic behavior
\begin{align}
	S(z)=I+O(1/n)
	\label{S_asy}
\end{align}
as $n\rightarrow \infty$ uniformly for $z\in\mathbb{C}\setminus \Gamma$.
	
\section{Main Results}

Let $X_{ij}$ denote the $(i,j)$ element in the matrix $X$.
For convenience, we define $m(z):=N_{11}(z)$ and $m^*(z):=-iN_{12}(z)h^{-1}(z)$. By \eqref{N} and
\r{h}, a straightforward calculation gives
\begin{align}
	m(z)=
	\frac{ (z-a)^{1/2}+ (z-b)^{1/2} }{ 2(z-a)^{1/4} (z-b)^{1/4} }
	\(
	\frac{  z+c/2+{(z-a)^{1/2}(z-b)^{1/2} } }{  2z }
	\)^{\a}
	\(
	\frac{  1-z+c/2-(z-a)^{1/2}(z-b)^{1/2}  }{  2 (1-z) }
	\)^{\b}
	\label{m}
\end{align}
and
\begin{align}
	m^*(z):=
	\frac{ (z-b)^{1/2}-(z-a)^{1/2} }{ 2(z-a)^{1/4} (z-b)^{1/4} }
	\(
	\frac{  z+c/2-(z-a)^{1/2}(z-b)^{1/2}  }{  2z }
	\)^{\a}
	\(
	\frac{  1-z+c/2+(z-a)^{1/2}(z-b)^{1/2}   }{  2 (1-z) }
	\)^{\b},
 \label{ms}
\end{align}
where $a$ and $b$ are given in \r{MRS} and $c=n/N$. Note that $M(z)$ in \r{M} is analytic in
$\mathbb{C}\setminus[a,b]$ and hence $m(z)$ is also analytic in $\mathbb{C}\setminus[a,b]$.
However, this is not true with the function $m^*(z)$, which has additional cuts $(-\infty,0]$ and
$[1,+\infty)$.
We finally  arrive at the following theorem:
\begin{theorem}
Let {\upshape I, II, III} be the regions shown in Figure \ref{asy_fig}, and let $l$ denote the Lagrange constant given in (\ref{phi_def}).  With  $\phi(z)$ and $\wt D(z)$ defined in (\ref{phi_def}) and (\ref{Dt}), the asymptotic formula of the polynomial $\pi_{N,n}(z)$ is given by
\begin{align}
\pi_{N,n}(z)=
   e^{nl/2} \left\{ \wt D(z)e^{ -n \phi(z)  }m(z)
\left[   1+O(\frac1n) \right] +\delta(n)
 	\right\}
	\label{p1}
\end{align}
for $z\in\mbox{\upshape I}$, where $\delta(n)=0$ for $z \in \mbox{\upshape I}
\setminus \Omega_\pm$ and
$\delta(n)$ is exponentially small in comparison with its leading term for $z\in\mbox{\upshape I}\cap\Omega_\pm$; see Figure 1. More precisely,
$\delta(n)= O(N^{ \max(\a,\b) } e^{ n\phi(z) } )$,
where $\Re \phi(z)$ is negative when $\Re z$ is bounded away from the interval $(a,b)$;
see Figure \ref{phi_fig}.

Let $\widetilde{f}(z)$  be defined  as in (\ref{ft_def}). We have
\begin{align}
\pi_{N,n}(z)=(-1)^n \sqrt{\pi} e^{n l/2 }
\left\{ \widetilde{\text{A}}(z,n)\left[ 1+O(\frac1n)  \right]
+ \widetilde{\text{B}}(z,n)\left[ 1+O(\frac1n)  \right]  \right\}
	\label{p2}
\end{align}
for $z\in\mbox{\upshape II}$, where
\begin{align*}
\widetilde{\text{A}}(z,n)
&=  [ n^{2/3} \ft(z)  ]^{1/4}
  [m(z)+m^*(z)]
\left\{
\sin(N\pi z) \Ai(n^{2/3}\widetilde{f}(z)) +\cos(N\pi z)\Bi(n^{2/3}\widetilde{f}(z))
\right\},
\end{align*}
and
\begin{align*}
\widetilde{\text{B}}(z,n)
&=
[ n^{2/3} \ft(z)  ]^{-1/4}
  [m(z)-m^*(z)]
\left\{
\sin(N\pi z) \Ai'(n^{2/3}\widetilde{f}(z)) +\cos(N\pi z)\Bi'(n^{2/3}\widetilde{f}(z))
\right\}.
\end{align*}

Similarly, with $\fs(z)$ defined  in (\ref{fs_def}),
\begin{align}
\pi_{N,n}(z)=(-1)^N  \sqrt{\pi} e^{n l/2 }
\left\{ {\text{A}}^*(z,n)\left[ 1+O(\frac1n)  \right]
+ {\text{B}}^*(z,n)\left[ 1+O(\frac1n)  \right]  \right\}
	\label{p3}
\end{align}
for $z\in\mbox{\upshape III}$,
where
\begin{align*}
{\text{A}}^*(z,n)
&=
[ n^{2/3} \fs(z)  ]^{1/4}
  [m(z)-m^*(z)]
\left\{
\cos(N\pi z)\Bi(n^{2/3}{f}^*(z))-\sin(N\pi z) \Ai(n^{2/3}{f}^*(z))
\right\},
\end{align*}
and
\begin{align*}
{\text{B}}^*(z,n)
&=
[ n^{2/3} \fs(z)  ]^{-1/4}
  [m(z)+m^*(z)]
\left\{
\cos(N\pi z)\Bi'(n^{2/3}{f}^*(z))-\sin(N\pi z) \Ai'(n^{2/3}{f}^*(z))
\right\}.
\end{align*}

\end{theorem}

\begin{proof}

Since
$
    R(z)=(c_1 e^{nl})^{\sigma_3/2} S(z) (c_1 e^{nl })^{-\sigma_3/2}\widetilde{R}(z)
$ by \r{S_def},
we have
\begin{align}
    R_{11}(z)=S_{11}(z)\widetilde{R}_{11}(z)
    +S_{12}(z)\widetilde{R}_{21}(z)c_1 e^{nl}
	\label{R11}
\end{align}
and
\begin{align}
    R_{12}(z)=S_{11}(z)\widetilde{R}_{12}(z)
    +S_{12}(z)\widetilde{R}_{22}(z)c_1 e^{nl}
.
 \label{R12}
\end{align}
From \eqref{S_asy}, it follows that
	\begin{align}
		S_{11}(z)=1+O(1/n)
		\qquad \mbox{and}
		\qquad
		S_{12}(z)=O(1/n).
	 \label{S:asy}
	\end{align}

First, let us consider $z\in\mbox{I}$. Recalling the definition of
$\wt R(z)$ in \eqref{Rt}, one obtains
	\begin{align}
		R_{11}(z)=
		e^{nl/2} e^{-n\phi(z) }
		{\wt D(z) }{  \prod_{j=0}^{N-1}(z-x_{N,j} )^{-1} }
		\left[
		S_{11}(z) N_{11}(z)+ S_{12}(z)N_{21}(z)
	 \right]
		\label{R11:p1}
	\end{align}
and
	\begin{align}
		R_{12}(z)
		= c_1 e^{nl/2} e^{n\phi(z)}
		{\wt D(z)^{-1} }{ \prod_{j=0}^{N-1}(z-x_{N,j} ) }
	\left[
		S_{11}(z) N_{12}(z)+S_{12}(z)N_{22}(z)
	\right]	.
		\label{R12:p1}
	\end{align}
Note by Theorem 2.1 and (3.1) that
	\begin{align}
		\pi_{N,n}(z)=Y_{11}(z)
		= H_{11}(z) \prod_{j=0}^{N-1}(z-x_{N,j} ).
				\label{pi:H}
	\end{align}
From (3.3)-(3.4), we know that $H_{11}(z)$ has different expressions in
different parts of the region I.
From \r{R_H}, \r{R11:p1}  and \r{S:asy}, it follows that
	\begin{align}
		\pi_{N,n}(z)
		&=R_{11}(z)\prod_{j=0}^{N-1}(z-x_{N,j} )
	  =
	 	e^{nl/2} {\wt D(z) }
	 	e^{-n\phi(z) }m(z)
		\left[
		1+O(1/n)
	 \right]
		\label{p1:nom}
	\end{align}
for $z\in\mbox{I}\setminus\Omega_\pm$,
since $m(z):=N_{11}(z)$.

Recalling the notation $c_1=2N\pi i$ in \r{r}, we have from (3.3) and \r{S:asy}-\r{R12:p1}
	\begin{align}
		\pi_{N,n}(z)
		&=R_{11}(z)\prod_{j=0}^{N-1}(z-x_{N,j} )
		-R_{12}(z) \frac{  \mp i e^{\pm N\pi i z } \cos(N\pi z)	}
		{  N\pi w(z)  \prod_{j=0}^{N-1}(z-x_{N,j} )	}
	 \hh
	 	e^{nl/2}
	 	\left\{
	 		e^{-n\phi} \wt D(z)N_{11}(z)
	 		[1+O(1/n)]
	 		\mp
	 		e^{n\phi(z)}
		\frac{  2  e^{\pm N\pi i z } \cos(N\pi z)	}
		{   \wt D(z)w(z) 	}N_{12}(z)
	\left[
		1+O(1/n)
	\right]	
	 	\right\}
		\label{p1:ome:pre}
	\end{align}
for $z\in\mbox{I}\cap\Omega_\pm$. For simplicity, we only consider $\Re z <x_0$.
It follows from \r{Dt} and \r{Ds_D} that
\begin{align}
		\pi_{N,n}(z)
		&=
	 	e^{nl/2}
	 	\left\{
	 		e^{-n\phi} \wt D(z)N_{11}(z) [1+O(1/n)]
	 		\mp
	 		e^{n\phi(z)}
		{    D(z)^{-1}w(z)^{-1} 	}N_{12}(z)
	\left[
		1+O(1/n)
	\right]	
	 	\right\}
	.
		\label{p1:ome}
	\end{align}
Note that $\Re \phi(z)<0$ for $z\in\mbox{I}$; see Figure \ref{phi_fig}.
Moreover, it is readily seen from \r{D_asy} that $D(z)$ is bounded
as $n\ra \infty$ for $z\in\mbox{I}\cap \Omega_\pm$. From \r{w_asy}. we
also have $w(z)$ bounded as $z\ra\infty$ for $z\in \mbox{I}\cap\Omega_\pm$ and $z\neq 0$.
Near $z=0$, $w(z)^{-1}=O(N^{\a})$. Thus, in terms of the notation $m(z)=N_{11}(z)$,
\begin{align}
		\pi_{N,n}(z)
		&=
	 	e^{nl/2}
	 	\left\{
	 		e^{-n\phi} \wt D(z)m(z) [1+O(1/n)]
	 		+
	 		O( N^{\a} e^{n \Re \phi(z) } )
	 	\right\}
	\end{align}
for $z\in\mbox{I}\cap\Omega_\pm$ and $\Re z<x_0$. Similarly, we can show
\begin{align}
		\pi_{N,n}(z)
		&=
	 	e^{nl/2}
	 	\left\{
	 		e^{-n\phi} \wt D(z)m(z) [1+O(1/n)]
	 		+
	 		O( N^{\b} e^{n \Re \phi(z) } )
	 	\right\}
	\end{align}
for $z\in\mbox{I}\cap\Omega_\pm$ and $\Re z>x_0$. Coupling the above two formulas with \r{p1:nom} gives \r{p1}.

Next, we consider the case $z\in \mbox{II}$; see Figure 6.
Let us first restrict $z\in \mbox{II}\cap \cp$.
From (7.6) and (7.7), we obtain by
a combination of (\ref{Etn}), (\ref{P}) and (\ref{Rt})
\begin{align}
    R_{11}(z)
    &=
    S_{11}(z)\rt_{11}(z)+S_{12}(z)\rt_{21}(z) c_1e^{nl}
    \hh
    (-1)^n\sqrt{\pi} e^{nl/2}
    \frac{ 2\cos(N\pi z)  }{ \prod_{j=0}^{N-1}(
          z-x_{N,j} ) }
    \bigg\{
	[-i\xt^{-\frac14}  \om^2 \Ai'(\om\xt)
	 -i \xt^{\frac14} \om \Ai(\om \xt)  ] N_{11}(z) [1+ O(1/n) ]
    \hhh
	\quad \qquad \qquad \qquad \qquad \qquad +
	[\xt^{-\frac14}  \om^2 \Ai'(\om\xt)
	-\xt^{\frac14} \om \Ai(\om \xt)
	] N_{12}(z) h(z)^{-1} [1+ O(1/n) ]
 	\bigg\}
 		\label{R11:p2}
\end{align}
and
\begin{align}
    R_{12}(z)&=
    S_{11}(z)\rt_{12}(z)+S_{12}(z)\rt_{22}(z) c_1e^{nl}
    \hh
   (-1)^n\sqrt{\pi}
    c_1 e^{nl/2}  \frac{ \prod_{j=0}^{N-1}(z-x_{N,j} ) }{ 2\cos(N\pi z)  }
    \bigg\{
    [i \xt^{-\frac14} \Ai'(\xt) +i \xt^{\frac14}\Ai(\xt)]
     N_{11}(z) h(z)[ 1 +O(1/n) ]
    \hhh
   \qquad \qquad \qquad \qquad \qquad \qquad \qquad \qquad  +
    	[\xt^{\frac14}\Ai(\xt)- \xt^{-\frac14} \Ai'(\xt)]
    	N_{12}(z) [1+O(1/n)  ]
    \bigg\}
		\label{R12:p2}
.
\end{align}
Here, use has also been made  of (7.8). Since $c_1=2 N\pi i$,
by \r{pi:H} and \r{R_H_ort} we have
\begin{align*}
		\pi_{N,n}(z)
		&=R_{11}(z)\prod_{j=0}^{N-1}(z-x_{N,j} )
		-R_{12}(z) \frac{  - i e^{\pm N\pi i z } \cos(N\pi z)	}
		{  N\pi w(z)  \prod_{j=0}^{N-1}(z-x_{N,j} )	};
\end{align*}
see the first equality in (7.13). A combination of this with (7.17) and (7.18) gives
\begin{align}
	   	   \pi_{N,n}(z)
	   &=
	   	(-1)^n\sqrt{\pi} e^{nl/2}
	 	\Bigg\{
	 	{ 2\cos(N\pi z)  }
    \bigg[
	[-i\xt^{-\frac14}  \om^2 \Ai'(\om\xt)
	 -i \xt^{\frac14} \om \Ai(\om \xt)  ] N_{11}(z) [1+ O(1/n) ]
    \hhh
	\qquad \qquad \ \qquad \qquad \qquad \qquad \ +
	[\xt^{-\frac14}  \om^2 \Ai'(\om\xt)
	-\xt^{\frac14} \om \Ai(\om \xt)
	] N_{12}(z) h(z)^{-1} [1+ O(1/n) ]
 	\bigg]
  \hhh	
 	\qquad \qquad \qquad \quad \ -
 	e^{N\pi i z} w(z)^{-1}
 	\bigg[
    [i \xt^{-\frac14} \Ai'(\xt) +i \xt^{\frac14}\Ai(\xt)]
     N_{11}(z) h(z)[ 1 +O(1/n) ]
    \hhh
   \qquad \qquad \qquad \qquad \qquad  \qquad \quad \quad \quad  +
    	[\xt^{\frac14}\Ai(\xt)- \xt^{-\frac14} \Ai'(\xt)]
    	N_{12}(z) [1+O(1/n)  ]
    \bigg]
 	\Bigg\}
	\label{p2:pre}
	\end{align}	
for $z\in\mbox{II}\cap\cp$. Recall the well-known formula of the Airy functions \cite[(9.2.11)]{Handbook}
\begin{eqnarray}
    \text{Bi}(z)
=
    \pm i \left[ 2 e^{\mp \pi i/3} \text{Ai}(\omega^{\pm 1} z)-\text{Ai}(z) \right]
.
\label{Bi}
\end{eqnarray}
Also, note that $w(z)= h(z)[1+O(1/n)]$ as $n\ra \infty$; see \r{w_asy} and \r{h}.
This together with \r{Bi} gives
	\begin{align}
		\pi_{N,n}(z)=(-1)^n \sqrt{\pi} e^{n l/2 }
\left\{ \widetilde{\text{A}}(z,n)\left[ 1+O(\frac1n)  \right]
+ \widetilde{\text{B}}(z,n)\left[ 1+O(\frac1n)  \right]  \right\}
		\label{p2:p}
	\end{align}
for $z\in\mbox{II}\cap\cp$.

In a similar manner, one can see that
\begin{align}
		\pi_{N,n}(z)
		&=
	   	(-1)^n\sqrt{\pi} e^{nl/2}
	 	\Bigg\{
	 	{ 2\cos(N\pi z)  }
    \bigg[
	[i\xt^{-\frac14}  \om \Ai'(\om^2\xt)
	 +i \xt^{\frac14} \om^2 \Ai(\om^2 \xt)  ]\cdot N_{11}(z) [1+ O(1/n) ]
    \hhh
	\qquad \qquad \ \qquad \qquad \qquad \qquad \ +
	[-\xt^{-\frac14}  \om \Ai'(\om^2 \xt)
	+\xt^{\frac14} \om^2 \Ai(\om^2 \xt)
	]\cdot N_{12}(z) h(z)^{-1} [1+ O(1/n) ]
 	\bigg]
  \hhh	
 	\qquad \qquad \qquad \quad \ +
 	e^{-N\pi i z} w(z)^{-1}
 	\bigg[
    [i \xt^{-\frac14} \Ai'(\xt) +i \xt^{\frac14}\Ai(\xt)]
     N_{11}(z) h(z)[ 1 +O(1/n) ]
    \hhh
   \qquad \qquad \qquad \qquad \qquad  \qquad \quad \quad \quad  +
    	[\xt^{\frac14}\Ai(\xt)- \xt^{-\frac14} \Ai'(\xt)]
    	N_{12}(z) [1+O(1/n)  ]
    \bigg]
 	\Bigg\}
	\label{p2:pre1}
	\end{align}
for $z\in\mbox{II}\cap\cm$. Again by \r{Bi}, we obtain the exactly  same formula given in \r{p2:p}, thus proving \r{p2}.
Following the same argument as given above, one can establish \r{p3} for $z\in \mbox{III}$. This completes the proof of Theorem 7.1.

\end{proof}

\section{Special Cases}

In this section, we want to investigate the asymptotic behavior of the Hahn polynomials
$Q_n(x;\a,\b,N-1)$
for fixed values of $x$.
First, we introduce the notation
	\begin{align}
		x:=Nz-1/2.
		\label{x}
	\end{align}
Then, it follows from (1.6) and (2.14) that
	\begin{align}
		Q_n(x;\a,\b,N-1)
		&=
		Q_n(Nz-1/2;\a,\b,N-1)
		\hh
		\frac{ N^n (n+\a+\b+1)_n }{ (\a+1)_n (-N+1)_n }
		\pi_{N,n}(z).
		\label{Q:exp}
	\end{align}

Now, we derive from \r{p1} asymptotic formulas for $Q_n(x;\a,\b,N-1)$ when $x$ is a fixed number (i.e., $z=O(1/N)$). For $x>-1/2$, i.e., $z>0$, we obtain from
\r{p1}, \r{Dt}, \r{Ds_D} and \r{pt_phi}
	\begin{align}
		\pi_{N,n}(z) \sim
		(-1)^n 2e^{ n(l/2- \pt(z))  }D(z) \cos(N\pi z) m(z).
		\label{pi:x:l0}
	\end{align}
\newpage
\noindent
Here, we have made use of the fact that $e^{-n\phi_\pm (z)}D^*_\pm(z)=(-1)^n 2\cos(N\pi z )
e^{-n\pt(z)} D(z) $ .

Substituting \r{x} into \r{D} gives
\begin{align}
	D(z)\cos(N\pi z)=\frac{ e^{Nz} \Gamma(Nz+1/2) }{ \sqrt{2\pi}(Nz)^{Nz} }\cos(N\pi z)=-\frac{ e^{x+1/2} \Gamma(x+1) }{ \sqrt{2\pi}(x+1/2)^{x+1/2} }
\sin(\pi x)	
.
\label{Dx}
\end{align}
Moreover, it is readily verified that
\begin{align}
	\frac{ N^n (n+\a+\b+1)_n }{ (\a+1)_n (-N+1)_n }
	\sim
	(-N)^n\frac{ \Gamma(\a+1) \Gamma(N-n)  }{\sqrt \pi n^{\a+1/2 }
	\Gamma(N) } 2^{2n+\a+\b}
 \label{cn}
\end{align}
as $n\ra \infty$, and by \r{m}
\begin{align}
	m(z)
\sim
	\frac{(1+c)^{1/2}}{ 2\cdot 2^{-1/2} c^{1/2} }
	\(\frac{c+1}{2c}\)^{\a}  \(\frac{c+1}{2} \)^{\b}
=   	 \frac{  (1+c)^{\a+\b+1/2} }{  2^{\a+\b+1/2} c^{\a+1/2}  }
\label{mx}
\end{align}
as $z \rightarrow 0$. We now derive an explicit formula for $l/2-\pt(z)$.
From  (\ref{pt_phi}), \r{phi_map_0} and \r{phi_def},  one has
\begin{align}
	l/2-\pt(z)
	=l/2-\phi_+(z)-\pi i(1-\frac1c z)
	=
	\text{Re}\,g_+(z)
\label{pt:gp}
\end{align}
for $z\in(0,a)$.
Note that since $x$ is fixed, by \r{x} $z\ra0$ as $n\ra \infty$. On account of \r{pt:gp} and  (4.58), by letting $n\ra \infty$,
 we obtain
\begin{align}
	l/2-\pt(z)=
	\left[
		-1-2\log 2 +(1+\frac1c)\log(1+c)
		\right]
		+z
		\left[
		\frac1c \log(z) -\frac2c \log c -\frac1c
		\right]+O(z^2),
	\label{pt:x:asy}
\end{align}
and so
\begin{align}
	e^{n(l/2-\pt(z))}
	\sim
	e^{-n}2^{-2n} (1+c)^{n+N}(x+1/2)^{x+1/2}e^{-x-1/2}N^{x+1/2}n^{-2x-1}
.
\label{pt:x:asye}
\end{align}
(The last two equations have also been used in [10, (6.27) \& (6.28)].)
A combination of  \r{pi:x:l0}, \r{Dx}, \r{mx} and \r{pt:x:asye} gives
\begin{align}
	Q_n(x;\a,\b,N-1) \sim
	-\frac{ \Gamma(\a+1) \Gamma(N-n)  }{\pi \Gamma(N) e^n n^{2x+2\a+2}}
	(1+c)^{n+N+\a+\b+1/2}
	N^{x+n+\a+1}\Gamma(x+1) \sin(\pi x)
	\label{Q:1}
\end{align}
for $x>-1/2$.

Next, we consider the case when $x<-1/2$, i.e., $z<0$.
From \r{g_def} and \r{phi_def}, it can be shown that
$e^{n\phi(z)}$ is analytic in the interval $(-\infty,0)$. Hence,  we obtain from
\r{p1} and \r{Dt}
\begin{align}
	\pi_{N,n}(z)\sim
    D^*(z)e^{ n (l/2- \phi(z))  }m(z).
    \label{pi:x:g0}
\end{align}
Note that
$e^{n(l/2-\phi(z) )}=e^{n(l/2-\phi_+(z)  )}= e^{ng_+(z)}$.
On account of  \r{g}, we have
\begin{align}
	g_+(z)=\left[\pi i
		-1-2\log 2 +(1+\frac1c)\log(1+c)
		\right]
		+z
		\left[
		\frac1c \log(-z) -\frac2c \log c -\frac1c
		\right]+O(z^2)
\label{Pp}
\end{align}
as $z\rightarrow 0$.
Letting $n\rightarrow \infty$, we obtain
\begin{align}
	e^{n(l/2-\phi(z))}
	\sim
	(-1)^n 2^{-2n}(1+c)^{n+N}e^{-n}
	 c^{-2x-1}N^{-x-1/2} (-x-1/2)^{x+1/2} e^{-x-1/2}
;
\label{pt:x1:asye}
\end{align}
equations \r{Pp} and \r{pt:x1:asye} are also given in [10, (6.33) \& (6.34)].
Inserting \r{x} in (\ref{Ds}) yields
\begin{align}
	D^*(z)= \frac{  \sqrt{2\pi}e^{x+1/2}  }{ (-x-1/2)^{x+1/2}\Gamma(-x)   }
.
\label{Dsx}
\end{align}
Hence, it follows from \r{Q:exp}, \r{pi:x:g0}, (\ref{mx}) and \r{pt:x1:asye}-(\ref{Dsx}) that
\begin{align}
	Q_{n}(x;\a,\b,N-1)\sim
	\frac{ \Gamma(\a+1) \Gamma(N-n)  }
	{ \Gamma(N)\Gamma(-x)  e^n n^{2x+2\a+2} }
	(1+c)^{n+N+\a+\b+1/2}
	N^{x+n+\a+1}
 \label{Q:2}
\end{align}
for $x<-1/2$.

To conclude this chapter, we make a comparison of our results with those presented in \cite{LinWong} for the discrete Chebyshev polynomials $t_n(z,N)$. First,
note that by taking $\a=\b=0$, the weight function $\rho(x;\a,\b,N)$ in \r{rho} becomes 1.
Furthermore, we have
\begin{equation}
	t_n(z,N)=(-1)^n(N-n)_n Q_{n}(z;0,0,N-1);
	\label{t:Q}
\end{equation}
see \cite[p.174 \& p.176]{BealsWong} and \cite{PanWong}.
For convenience, we consider only the case when $z$ is real and $\Re z<x_0$. By \r{m} and \r{ms}, we have
\begin{equation}
	m(z)= \frac{ (z-a)^{1/2}+(z-b)^{1/2}  }{ 2(z-a)^{\frac14} (z-b)^{\frac14} }
	\mbox{\quad  and \quad}
	m^*(z)= \frac{ (z-b)^{1/2}-(z-a)^{1/2} }{ 2(z-a)^{\frac14} (z-b)^{\frac14} }.
	\label{m:1}
\end{equation}
If $\wt \pi_{N,n}(z)$ denotes the monic polynomials associated with the discrete Chebyshev polynomials, then we have the relationship
$$
	\wt \pi_{N,n}(Nz-\frac12)= \frac{ n!^2 }{  (2n)! }
	t_n(Nz-\frac12,N) =N^n \pi_{N,n}(z),
$$
in view of \r{Q:exp}, \r{t:Q} and \cite[(2.1)]{LinWong}.
Since the function  $g(z)$ in \r{g_def} does not depend on $\a$ and $\b$, it  is the same  $g$-function defined in \cite[(3.4)]{LinWong}; also see \r{g} and \cite[(6.25)]{LinWong}.
For $z\in \mbox{II}$ (that is, $x_1< z<x_0$), in the notations of this paper we can rewrite (5.1) in \cite{LinWong} as
\begin{align}
    \pi_{N,n}(z)
&\sim
    (-1)^n\sqrt{\pi} e^{nl/2}
  \hhh
  \times
     \Bigg\{
        \left[
        \sin(N\pi z) \text{\upshape Ai}(n^{2/3}\widetilde{f}(z))+\cos(N\pi z)\text{\upshape Bi}(n^{2/3}\widetilde{f}(z))
        \right]
        \frac{(z-b)^{1/4} }{ (z-a)^{1/4}   }[n^{2/3}\widetilde{f}(z)]^{1/4}
\nonumber
\\
&\qquad \quad+
        \left[
        \sin(N\pi z) \text{\upshape Ai}'(n^{2/3}\widetilde{f}(z))+\cos(N\pi z)\text{\upshape Bi}'(n^{2/3}\widetilde{f}(z))
        \right]
        \frac{ (z-a)^{1/4} }{ (z-b)^{1/4}  }[n^{2/3}\widetilde{f}(z)]^{-1/4}
    \Bigg\},
\label{pi_left}
\end{align}
on account of  (3.37) in \cite{LinWong}.
In view of \r{m:1}, it can be readily  seen  that \r{pi_left} agrees with \r{p2}.
Moreover,
we obtain from  \cite[(6.2)]{LinWong}, \cite[(3.34)]{LinWong}, \r{Ds} and \r{m:1}
\begin{align}
	\pi_{N,n}(z)
&\sim
	 e^{nl/2}
 	 e^{ -n\phi(z) } \frac{ \sqrt{2\pi} e^{Nz}  (-Nz)^{-Nz} }
 	 { \Gamma(-Nz+1/2) }
	\frac{  (z-a)^{1/2}+(z-b)^{1/2}  }{ 2(z-a)^{1/4}(z-b)^{1/4}   }
	= D^*(z)e^{ n (l/2- \phi(z))  }m(z)
	\label{t18}
\end{align}
for $z<0$. Similarly, by \cite[(6.1)]{LinWong}, \cite[(3.33)]{LinWong}, \r{D} and \r{m:1} we have
\begin{equation}
	\pi_{N,n}(z) \sim
		(-1)^n 2e^{ n(l/2- \pt(z))  }D(z) \cos(N\pi z) m(z)
	\label{t19}
\end{equation}
for $0<z<x_1$. A comparison of \r{t18} and \r{t19} with \r{pi:x:g0} and \r{pi:x:l0} implies that our results  agree with those in \cite{LinWong}. In fact, one can easily show that when $\a=\b=0$, we do not need to construct the parametrix in region I. However, $\wt R(z)$ defined in (6.31) for $z\in\mbox{II}$ can not be extended to the region I when $\a,\b \neq 0$. This is because the function $h(z)$ involved in $\wt R(z)$ is not analytic in the neighborhoods of $0$ and $1$.


\end{document}